\documentclass[11pt]{article}
\usepackage{a4wide}
\usepackage{color,epsfig}
\usepackage{amsmath}
\usepackage{amsfonts}
\usepackage{dsfont}
\usepackage{mathrsfs}
\usepackage{amssymb}
\usepackage{amsfonts,graphicx,psfrag}
\usepackage{cite}
\usepackage{amsthm}
\usepackage{hyperref}
\usepackage{enumerate}
\usepackage{mathtools}
\usepackage{enumitem}
\usepackage{xcolor}
\mathtoolsset{showonlyrefs}

\newtheorem{Th}{Theorem}[section]

\newtheorem{lemma}{Lemma}[section]
\newtheorem{theorem}[lemma]{Theorem}

\newtheorem{proposition}[lemma]{Proposition}
\newtheorem{definition}[lemma]{Definition}

\newtheorem{remark}[lemma]{Remark}
\let\lutzremark=\remark
\def\remark{\lutzremark\normalfont}

\newtheorem{notation}[lemma]{Notation}

\def\be{\begin{equation}}
\def\ee{\end{equation}}
\def\bea{\begin{eqnarray}}
\def\eea{\end{eqnarray}}
\def\bes{\begin{eqnarray*}}
\def\ees{\end{eqnarray*}}

\def\nn{\nonumber}
\def\<{\langle}
\def\>{\rangle}
\def\lb{\label}
\def\bs{\setminus}
\def\pt{\partial}
\def\d{{\mathrm{d}}}

\def\R{{\bf R}}
\def\C{{\bf C}}
\def\Z{{\bf Z}}

\def\U{{\bf U}}

\def\aa{{\alpha}}
\def\bb{{\beta}}
\def\ga{{\gamma}}

\def\th{{\theta}}

\def\om{{\omega}}
\def\Om{{\Omega}}

\def\lm{{\lambda}}

\def\sg{{\sigma}}

\def\vf{{\varphi}}

\def\cA{{\cal A}}

\def\P{{\cal P}}

\def\diag{{\rm diag}}

\def\Sp{{\rm Sp}}

\def\dm{{\rm \diamond}}

\def\ol#1{\overline{#1}}
\def\td#1{\tilde{#1}}

\title{Linear stability of the elliptic relative equilibria for the restricted $N$-body problem: two special cases}

\author{Jiashengliang Xie$^{1}$\thanks{E-mail:22135010@zju.edu.cn.},\quad
	Bowen Liu$^{2}$\thanks{Partially supported by NSFC (Nos. {12101394,  12171426}), Science and Technology Innovation Action Program of  STCSM (No. 20JC1413200), Natural Science Foundation of Shanghai No. 22ZR1433100	and Innovation Program of Shanghai Municipal Education Commission. E-mail: liubowen2010@gmail.com} \quad and \quad
	Qinglong Zhou$^{1}$\thanks{Corresponding author. Partially supported by NSFC (No.12171426), the Natural Science Foundation of Zhejiang Province (No. Y19A010072) and the Fundamental Research Funds for the Central Universities (No. 226-2024-00136).
		E-mail: zhouqinglong@zju.edu.cn. }\\	
	$^{1}$ Department of Mathematics,\\Zhejiang University, Hangzhou 310058, Zhejiang, China\\
	$^{2}$ School of Mathematical Science,\\ Shanghai Jiao Tong University, Shanghai 200240, China\\
}
\date{}

\begin{document}

\maketitle

\begin{abstract}
In this paper, we consider the elliptic relative equilibria of the restricted $N$-body problems, where
the $N-1$ primaries form an Euler-Moulton collinear central configuration 
or a $(1+n)$-gon central configuration. 
We obtain the symplectic reduction for the general restricted $N$-body problem. 
For the first case,
by analyzing the relationship between these restricted $N$-body problems and the elliptic Lagrangian solutions, we distinguish the linear stability of the restricted $N$-body problem by the $\om$-Maslov index.Through numerical computations, we also determine the stability conditions in terms of the mass parameters for \(N=4\) and the symmetry of the central configuration.
For the second case,
there exist three positions
$S_1,S_2$ and $S_3$
for the massless body (up to rotations of angle $\frac{2\pi}{n}$).
For ${m_0\over m}$ sufficiently large,
we show that the elliptic relative equilibria
is linearly unstable if the eccentricity $0\le e<e_0$ and the massless body lies at $S_1$ or $S_2$;
while the elliptic relative equilibria
is linearly stable if the massless body lies at $S_3$.
\end{abstract}



\renewcommand{\theequation}{\thesection.\arabic{equation}}

\setcounter{equation}{0}
\section{Introduction and main results}
\label{sec:1}

{In the classical planar $N$-body problems of celestial mechanics, the position vectors of the $N$-particles are denoted by $q_1 ,\dots , {q_N}\in \R^2$, and the masses are represented by
$m_1 ,\dots,{m_N} > 0$.} By Newton’s second law and the law of universal
gravitation, the system of equations is
\begin{align}
m_i\ddot{q_i} = \frac{\pt U}{\pt q_i}, \quad i = 1, {\dots, N},\lb{1.1}
\end{align}
where $U(q) = U(q_1, {\dots, q_{N}}) = \sum_{1\leq i< j\leq {N}} \frac{m_i m_j}{|q_i-q_j|}$ 
is the potential function and $|\cdot |$ {is} the standard norm of vector in $\R^2$. 
Suppose the configuration space is 
$$\hat{\chi} :=\left\{q =(q_1, {\dots, q_N})\in (\R^2)^{{N}}\;|\sum_{i = 1}^{{N}} m_iq_i = 0, q_i \neq q_j, \forall i\neq j \right\}. $$
For the period $T$, the corresponding action functional is 
\begin{align}
  \mathbf{A}(q) = \int_{0}^{T} \left[ \sum_{i = 1}^{{N}} \frac{m_i|\dot{q}_i(t)|^2}{2} +U(q(t))\right] \d t, \lb{1.222}
  \end{align}
which is defined on the loop space $W^{1,2} (\R/T \Z, \hat{\chi})$. The periodic solutions
of \eqref{1.1} correspond to critical points of the action functional \eqref{1.222}. 
Let $p_1 , {\dots, p_{N}} \in \R^2$ be the momentum vectors of the particles respectively. It is {well-known} that \eqref{1.1} can be reformulated as a Hamiltonian system by 
\begin{align}
\dot{p}_i = -\frac{\pt H}{\pt q_i}, \; \dot{q}_i  = \frac{\pt H}{\pt p_i}, \quad \mbox{for} \;i = 1, {\dots, N}, \lb{1.2}
\end{align}
with the Hamiltonian function
\begin{align} 
H(p, q) = \sum_{i =1}^{{N}} \frac{|p_i|^2}{2m_i} - U(q_1,{\dots, q_{N}}).\lb{1.3}
\end{align}

One special class of periodic solutions to the planar $N$-body problem is the elliptic relative equilibrium (ERE for short) \cite{MS}. It is generated by a central configuration and the Keplerian motion.
A central configuration (C.C. for short) is formed by $N$ position vectors $\left(q_{1}, \ldots, q_{N}\right)=\left(a_{1}, \ldots, a_{N}\right)$ which satisfy
\begin{align}
	-\lambda m_{i} q_{i}=\frac{\partial U}{\partial q_{i}}, \forall\; 1\leq i \leq N, \lb{eqn:cc}
\end{align}
where $\lambda =U(a) /  {I(a)}>0$ and $I(a)={\sum_{i=1}^N} m_{i}|a_{i}|^{2}$ is the moment of inertia.
A planar central configuration of the $N$-body problem gives rise to a solution of \eqref{1.1} where each particle moves on a specific Keplerian orbit while the totality of the particles move according to a homothetic motion.


The linear stability of the ERE is determined by the eigenvalues of the linearized Poincar\'e map. 
Let $\U$ denote the unit circle in the complex plane. The ERE is spectrally stable if all eigenvalues of linearized Poincar\'e map  are on $\U$; it is linearly stable if Poincar\'e map is semi-simple and spectrally stable; it is linearly unstable if at least one pair of eigenvalues are not on $\U$.
Since the nineteenth century \cite{R2}, the researches on the stability have always been active in celestial mechanics because it reveals the dynamics near the period orbits. 
However, it has always been one difficult task to obtain the linear stability of ERE, because the linearized Hamiltonian systems are non-autonomous, especially for the elliptic orbits.
Many results on linear stability of the three-body problems have been obtained over the past decades by numerical methods \cite{MSS,MSS1,MSS2}, bifurcation theory \cite{R1} and the index theory \cite{HS,HLS,HuOuWang2015ARMA,Zhou2017}.  
To the best of our knowledge, the $\om$-Maslov index theory is the only analytical method to obtain the full picture of the stability and instability to the ERE, such as the elliptic Lagrangian solution \cite{HS,HLS,HuOuWang2015ARMA}, and the elliptic Euler solution \cite{Zhou2017,ZhonLong2017CMDA}.

When it comes to $N$-body problems with $N \geq 4$, research on the stability of ERE is quite challenging, and hence results have been scarce in recent decades.
To the best of our knowledge,
for the general case of $n$ bodies, the elliptic Euler$-$Moulton solutions
are the only ones that has been well studied (in \cite{ZhonLong2017CMDA}).
In some special cases of $n$-body problem,
for the ERE which derived from an $n$-gon or $(1+n)$-gon central configurations with $n$ equal masses,
the linear stability problem has been studied by Hu, Long and Ou in \cite{Hu2020}.
For $N=4$, results 
regarding the linear stability of other ERE to the four-body problem can be found in \cite{Mansur2017, Leandro2018, Liu2021,LiZ, Leandro2003,zhou2023}.

In this paper, we focus on the restricted $N$-body problem with the primaries forming an Euler-Moulton collinear central configuration 
(see \cite{ZhonLong2017CMDA})
or a $(1+n)$-gon central configuration (see \cite{Maxwell,Hu2020}).
It is reasonable to study the stability problem of the
massless particle
since the
“massless body” can be imaged as a space station while the massive bodies form a planetary system.

We first apply the symplectic reduction method \cite{MS} to the general restricted $N$-body problem. 
Let $P$ denote the inertial position of the massless body, which moves in the gravitational field of the $N-1$ primaries without disturbing their motion. The corresponding 
 Hamiltonian function for this $N$-th body is given by
\begin{equation}
H(P,q, t)=\frac{1}{2}|P|^2-\sum_{i=1}^{N-1}
\frac{m_i}{|r(t)R(\th(t))a_i-q|}.\lb{HaFun}
\end{equation}
Here $r=r(\theta(t))$ is the Kepler elliptic orbit given through the true anomaly $\th=\th(t)$by
\begin{align}
  r(\th(t)) = \frac{p}{1+e\cos\th(t)},  \lb{rTh}
\end{align}
where $p=a(1-e^2)$ and $a>0$ is the latus rectum of the ellipse \eqref{rTh}.
By Proposition \ref{P2.1}, the ERE $(P(t),q(t))^T$  of the system \eqref{1.2} is in time $t$ where $q(t)=r(t)R(\theta(t))a_{N}$ and $P(t)=\dot{q}(t)$. It 
can be transformed to the new solution $(\bar{Z}(\theta),\bar{z}(\theta))^T = (0,\sigma,\sigma,0)^T$ in the true anomaly $\theta$ as the new
variable for the original Hamiltonian function $H$ given by \eqref{eqn:red.Ham} below. We then have the reduced Hamiltonian system is given in the theorem.

\begin{theorem}\label{linearized.Hamiltonian}
The linearized Hamiltonian system 
at the ERE
  $\zeta_0 \equiv (\bar{Z}(\theta),\bar{z}(\theta))^T =
  (0,\sigma,\sigma,0)^T\in\R^4  $
  depending on the true anomaly $\theta$ is given by
  \begin{align}
    \dot\xi(\theta) = JB(\theta)\xi(\theta),  \lb{eqn:LinearHam1}
  \end{align}
  with
  \bea B(\theta)
  = H''(\theta,\bar{Z},\bar{z})|_{\bar\xi=\xi_0}
  = \left(\begin{array}{cccc|cccc}
  I_2   &-J_2 \\
  J_2     &I_2-\frac{1}{1+e\cos\th}D
  \end{array}\right),  \lb{LinearHam2}\eea
  where
  \begin{align}
    \label{matrix.D}
  D=I_2-\frac{1}{\mu}\left(\sum_{i=1}^{N-1}
  \frac{m_i}{|a_i-a_{N}|^3}\right)I
  +\frac{3}{\mu}\sum_{i=1}^{N-1} m_i\frac{(a_i-a_{N})(a_i-a_{N})^T}{|a_i-a_{N}|^5}.
  \end{align}
  The corresponding quadratic Hamiltonian function is given by
  \bea
  H_2(\theta,\bar{Z},\bar{z})  
   =\frac{1}{2}|\bar{Z}|^2+\bar{Z}\cdot J\bar{z}
    +\frac{1}{2}\left(I_2-\frac{D}{1+e\cos\th}\right)|\bar{z}|^2. \eea
  \end{theorem}

Since $D$ depends on the central configuration,
it is challenging to compute $D$ for a general central configuration.
Therefore, we turn our attention to two special cases.
  
The first case is when the primaries form an Euler-Moulton collinear configuration and the $N$ bodies span $\R^2$.
Denote the eigenvalues of $D$ by $\lm_3$ and $\lm_4$. By Proposition  \ref{prop.sum.lm3.lm4}, both $\lm_3$ and $\lm_4$ are positive and  $\lm_3 + \lm_4 = 3$.
By this property, we establish the relationships between the Maslov index of $\xi_{\aa, e}$ given by \eqref{eqn:LinearHam1} and the Maslov index of $\ga_{\beta, e}$ which represents the elliptic Lagrangian solutions in \cite{HLS}. 
For more details on Maslov-type index theory, we refer to the Appendix in \ref{subsec:5.2} below, while all the complete details are contained in \cite{Lon4}.

Since the elliptic Lagrangian solutions have already been well-studied in \cite{HLS},
by the properties of the Maslov index, we define three curves $\bb_s(e)$, $\bb_m(e)$ and $\bb_k(e)$ from left to right in $(\bb, e)\in [0,9]\times [0, 1)$ according to the $\omega$-Maslov indices of $\xi_{\bb,e}$. 
Therefore, we can distinguish the linear stability
of the EREs as follows.

\begin{theorem}\label{thm:RE.norm.form}
Let $\lambda_3, \lambda_4$ be the two eigenvalues of
$D$ given by (\ref{matrix.D}), 
and $\beta=9-(\lambda_3-\lambda_4)^2$.
There exist two real analytic functions
$\beta_1(e),\beta_2(e),e\in[0,1)$
such that,
for every fixed $e\in[0,1)$,
the $-1$ index $i_{-1}(\xi_{\bb,e})$ is non-increasing
with respect to $\beta$;
and it is strictly decreasing on two values of $\bb=\bb_1(e)$ and $\bb=\bb_2(e)$.
We define
\begin{equation}
    \beta_s(e)=\min\{\beta_1(e),\beta_2(e)\},\quad
    \beta_m(e)=\max\{\beta_1(e),\beta_2(e)\},
\end{equation}
and
\begin{equation}\label{beta_h}
    \bb_k(e)=\sup\left\{\bb'\in[0,9]\;\bigg|\;\sigma(\xi_{\bb,e}(2\pi))\cap{\bf U}\ne\emptyset,\;
    \forall \bb\in[0,\bb']\right\},
\end{equation}
for $e\in[0,1)$.
The diagrams of
the functions $\bb_s$, $\bb_m$ and $\bb_k$ with respect to $e\in[0,1)$,
which denoted by $\Gamma_s,\Gamma_m$ and $\Gamma_k$,
separate the parameter rectangle $\Theta=[0,9]\times[0,1)$ into four regions,
which we denote by I, II, III and IV (see Figure \ref{bifurcation_curves}), respectively.
Then, if $(\bb,e)$ is in Region I and III, 
the ERE is strongly linearly stable;
if $(\bb,e)$ is in Region II and IV, 
the ERE is linearly unstable.

\begin{figure}[ht]
\centering
\includegraphics[height=7.5cm]{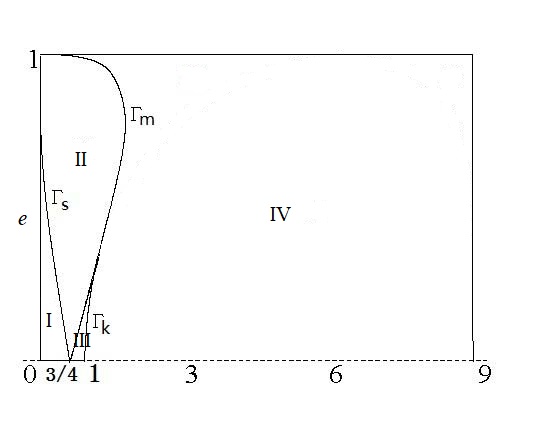}
\vspace{-5mm}
\caption{The linear stability separation curves with respect to the parameter region.}
\label{bifurcation_curves}
\end{figure}
\end{theorem}



However, the dependence between the mass parameter $\beta$ and the masses $m=(m_1,m_2,\ldots,m_{N-1})$
is quite intricate.
To shed light on how the linear stability varies with specific mass distributions, we will conduct some numerical computations.
For the circular case
with $N=4$, we first numerically analyze the relationships between $\bb$ and $m_1$, $m_2$, $m_3$, and consequently, demonstrate the linear stability in terms of $m_1$, $m_2$ and $m_3$. 
Since $m_1 + m_2 + m_3 = 1$, we can plot the stable regions in terms of $m_1$ and $m_3$.
We show these results in (a) of Figure \ref{fig:1}.
\begin{figure}[htbp]
	\centering
  \begin{tabular}{cc}
    \includegraphics[width=0.45\textwidth]{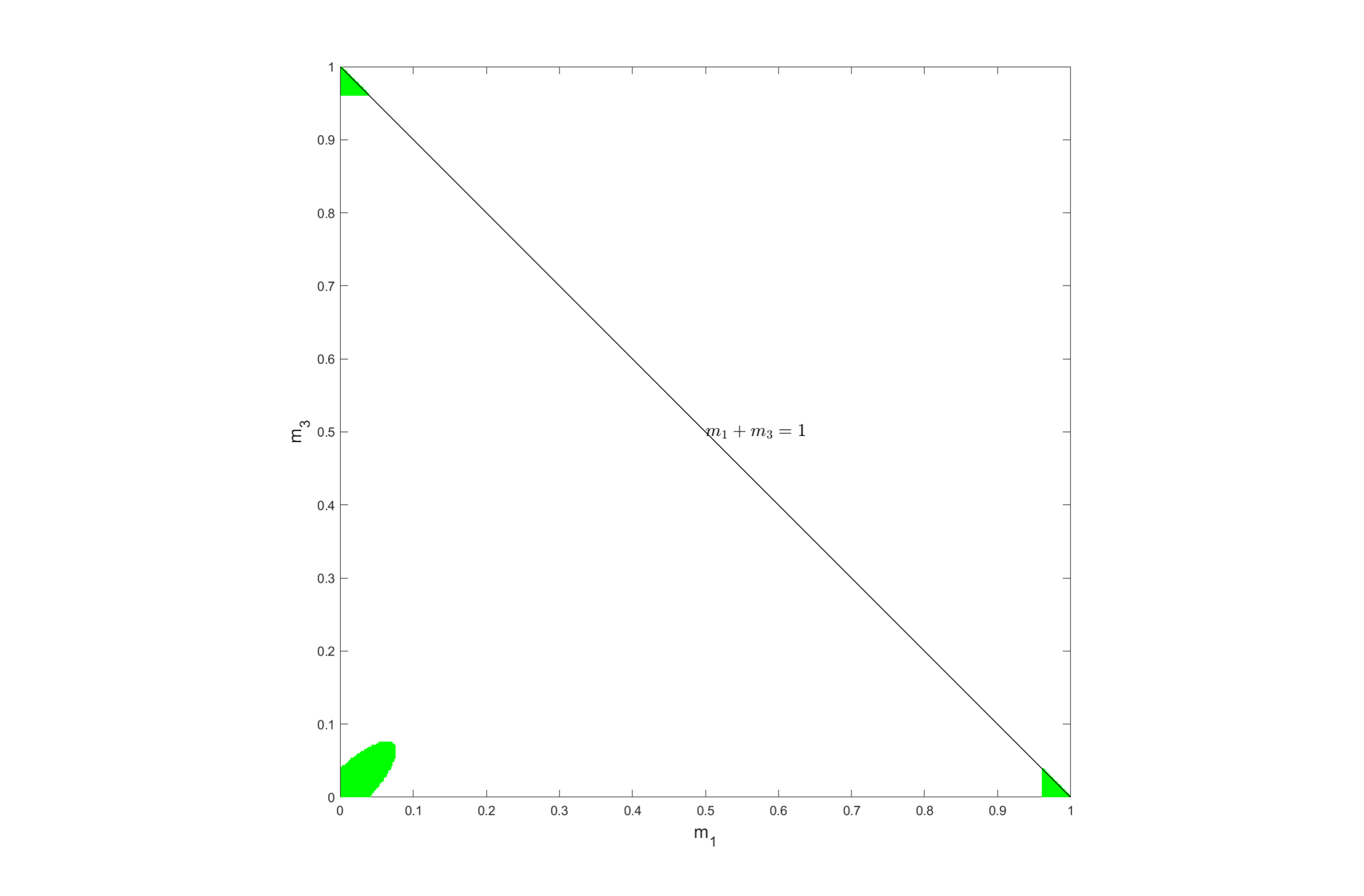} &
		\includegraphics[width=0.45\textwidth]{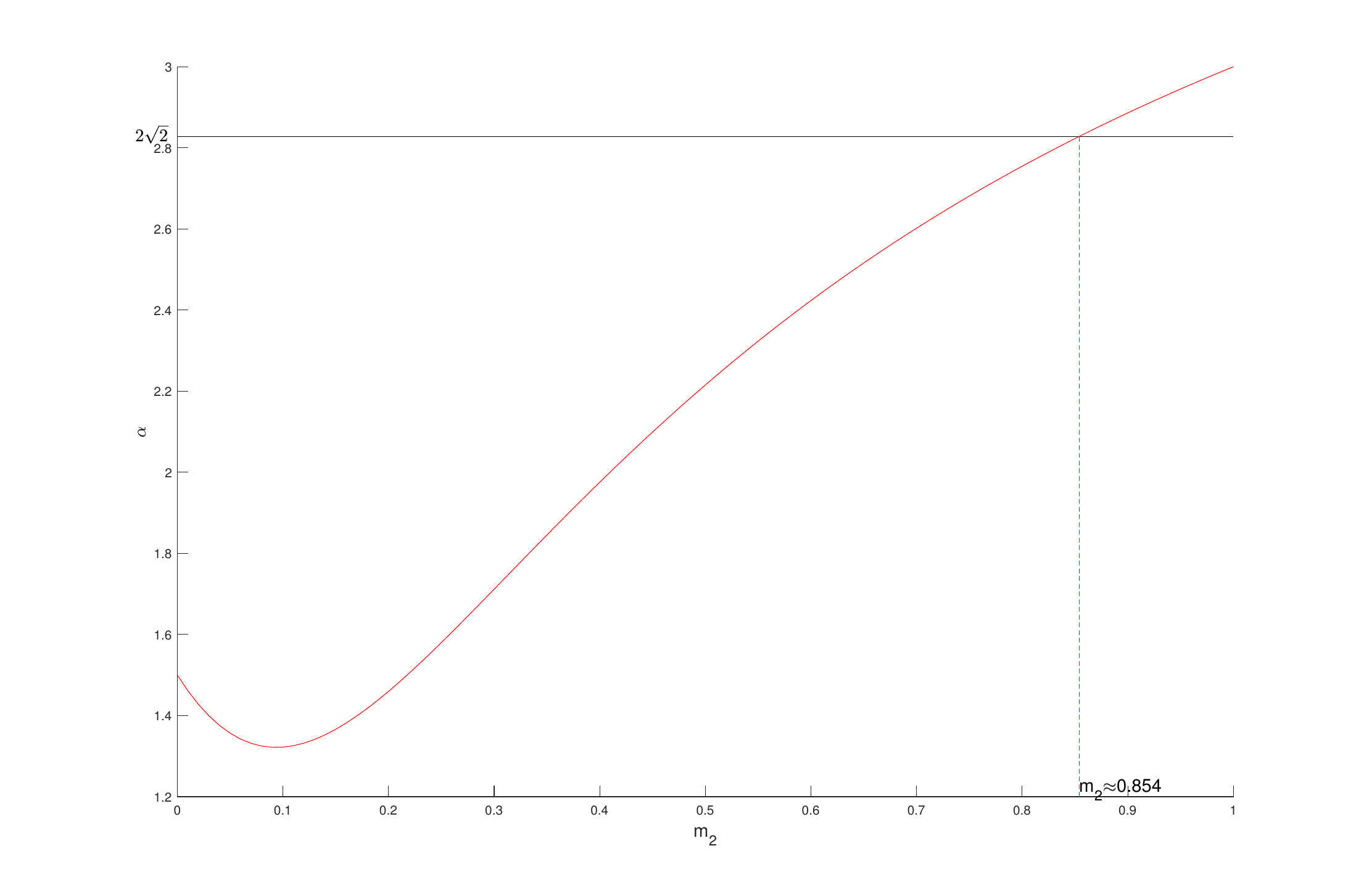}  \\
		\footnotesize{(a) The linearly stable regions of circular solution.}  & 
		\footnotesize{(b) The symmetric case.} 
	\end{tabular}
	\caption{\footnotesize{In Figure (a), the shade regions in the $(m_1,m_3)$-plane show the all the possible choices of $(m_1, m_2, m_3) = (m_1, 1- m_1 - m_3, m_3)$ of the circular orbit such that the system is linearly stable. The Figure (b) shows the changes of $\alpha$ with respect to $m_2$ in the symmetric case}}
	\label{fig:1}
\end{figure}

If we further assume that $m_1 = m_3$,
w
e have the following results:
\begin{theorem}\lb{thm:numerically}
  If $e = 0$ and $m_1 = m_3$, $\xi_{\bb,e}$ is linearly stable for $m_2 \in (m^*,1)$ for some $m^*\in(0,1)$.
\end{theorem}
Numerical computation shows that $m^*\approx0.854$,
see Figure \ref{fig:1} (b).

The second  case is when the primaries form a $(1+n)$-gon central configuration.

Suppose $n$ equal masses are placed at the vertices of a regular polygon, rotating around $m_0$ at its center.
Here note that $n=N-2$.
In \cite{BaE},
D.~Bang and B.Elmabsout determine the positions of relative equilibrium for the massless body lying
(up to rotations of angle $\frac{2\pi}{n}$)
on the semi-axes $\th=0$ and $\th=\frac{\pi}{n}$
(see Figure \ref{fig:2}):

\hangafter 1
\hangindent 3.5em
\;(i) On the semi-axis $\th=0$, there are two solutions
$\rho=S_1>1$, and $\rho=S_2,0<S_2<1$;

\hangafter 1
\hangindent 3.5em
(ii) On the semi-axis $\th=\frac{\pi}{n}$, there is one solution
$\rho=S_3>1$.

\noindent Here, $S_1,S_2$ and $S_3$ also denote the 
positions of the massless body, respectively.
In the same paper,
D.~Bang and B.Elmabsout have proved that $S_1$ and $S_2$
are spectrally unstable for every $m_0>0$,
while $S_3$ is linearly stable for $n\ge7$ if
$\frac{m_0}{m}$ is sufficiently large.
Their analysis was based on the methods which
analytically compute means of functions on regular $n$-gons
and study cyclic quantities of the complex variables in
\cite{BaE2}.
However, they just consider the circular orbits,
i.e., the eccentricity $e=0$.

\begin{figure}[htbp]
	\centering
    \includegraphics[width=10cm]{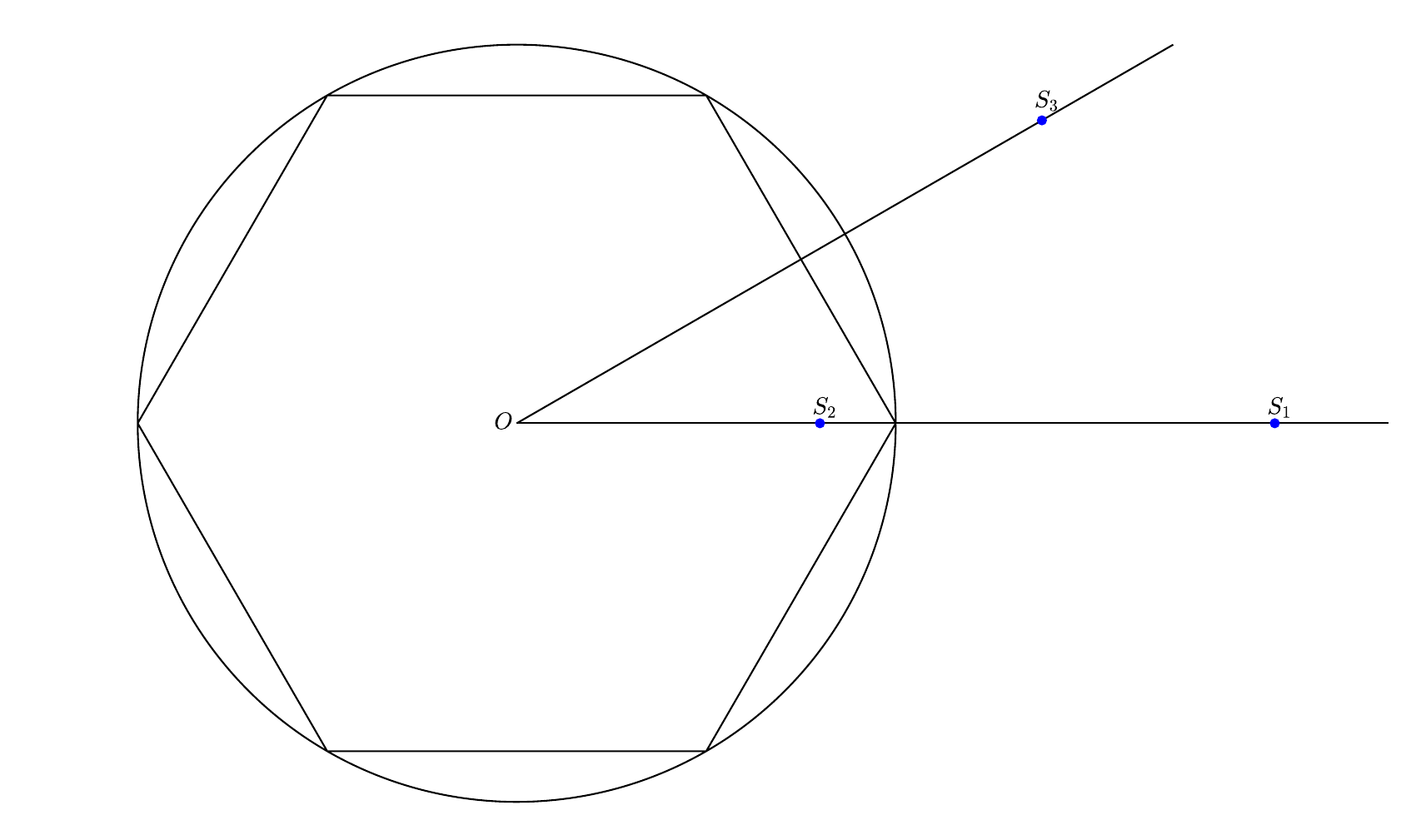}
	\caption{Positions of relative equilibrium for the massless body in the case $m_0>0$.}
	\label{fig:2}
\end{figure}

For the elliptic relative equilibrium
with eccentricity $e\in[0,1)$, we have
\begin{theorem}\label{no stability of S_1 and s_2 }
If the massless body lies at $S_1$ or $S_2$,
there exists an $e_0\in(0,1)$ such that
    the elliptic relative equilibrium is linearly unstable when ${m_0\over M}$ is sufficiently large and the eccentricity $e\in [0,e_0]$.
\end{theorem}
\begin{theorem}\label{stability of S_3}
If the massless body lies at $S_3$,
    the elliptic relative equilibrium is linear stability when ${m_0\over M}$ is sufficiently large.
\end{theorem}

This paper is organized as follows. We first introduce the generalized symplectic reduction method for the restricted $N$-body problem in Section \ref{sec:2}. We then use the $\om$-Maslov index theory to determine the linear stability of the restricted $N$-body problem
when the primaries form an Euler-Moulton central configuration in Section \ref{sec:3}. We also consider the symmetric case by assuming $m_1 = m_3$
of the restricted $4$-body problem in Section \ref{sec:3}. In Section \ref{sec:4}, we study the linear stability problem when the primaries form an $(1+n)$-gon central configuration.
Theorem \ref{no stability of S_1 and s_2 }
and Theorem \ref{stability of S_3}
are also proved there.

\setcounter{equation}{0}
\section{Reduction for the restricted \texorpdfstring{$N$}{N}-body problem}\label{sec:2}
The central configuration coordinates
for a class of periodic solutions to the $n$-body problem was introduced in \cite{MS}. 
In this section, we generalize this reduction to the restricted $N$-body problem.
For the given masses $m=(m_1,m_2,\ldots,m_{N-1})\in (\R^+)^{N-1}$
of $N-1$ primaries, 
let $a=(a_1,a_2,\ldots,a_{N -1})$ be an $(N - 1)$-body central configuration of $m$  satisfying $m_1+m_2+\ldots+m_{N-1}=1$.
Using normalization $2I(a) = \sum_{i = 1}^{N-1}m_i |a_i|^2 = 1$ and assuming  $\mu = U(a) = \sum_{1\le i<j\le N-1}\frac{m_im_j}{|a_i-a_j|}$, we have
\begin{align}
  \sum_{j=1,j\ne i}^{N-1}\frac{m_j(a_{j}-a_{i})}{|a_{j}-a_{i}|^3}
          = -\frac{U(a)}{2I(a)}a_{i}=-\mu a_{i}.    \label{eq.of.cc}
\end{align}

\begin{proposition}\label{P2.1}
There exists a symplectic coordinate change
$
  \xi = (P,q)^T
    \;\mapsto\; \bar{\xi}
    = ( \bar{Z}, \bar{z})^T,$
such that using the true anomaly $\th$ as the variable,  the resulting Hamiltonian function of the
$n$-body problem is given by
\begin{align}\lb{eqn:red.Ham}
  H(\theta,\bar{Z},\bar{z})=\frac{1}{2}|\bar{Z}|^2 + \bar{z}\cdot J\bar{Z}
+ \frac{p-r}{2p}|\bar{z}|^2
- \frac{r}{\sg}\sum_{i=1}^{N-1}{m_i\over|\sg a_i-\bar{z}|},
\end{align}
where $J =( 
\begin{smallmatrix}
       0 &-1 \\ 1 & 0
\end{smallmatrix})$, $r(\th)=\frac{p}{1+e\cos\th}$,
$\mu$ is defined by \eqref{eq.of.cc}, $\sg=(\mu p)^{1/4}$ and $p$ is given in \eqref{rTh}.
\end{proposition}
\begin{proof}
  Inspired by Lemma 3.1 of \cite{MS}, we carry the coordinate
  changes in four steps.
  
  {\bf Step 1.} {\it Rotating coordinates via the matrix $R(\th(t))$ in time $t$.}
  
  We first change the coordinates $\xi$ to $\hat{\xi}=(\hat{Z}, \hat{z})^T \in (\R^2)^2$
  which rotates with the speed of the true anomaly. The transformation matrix is given by the rotation
  matrix $R(\th) = (
  \begin{smallmatrix}
         \cos\th &-\sin\th \\\sin \th & \cos\th
  \end{smallmatrix})$. The generating function of this
  transformation is given by
  \be  \hat{F}(t, P, \hat{z})
         = -P\cdot R(\th)\hat{z},  \lb{P-02}\ee
  and for $1\le i\le n-1$ the transformation is given by
  \be
  q = -\frac{\pt \hat{F}}{\pt P} = R(\th)\hat{z},  \quad 
  \hat{Z} = -\frac{\pt \hat{F}}{\pt \hat{z}} = R(\th)^TP,  \lb{P-03}. \ee
  Writing $\dot{R}(\th(t))=\frac{d}{dt}R(\th(t))$, and noting that $R(\th)^T=R(\th)^{-1}$ and
  $\dot{R}(\th)=\dot{\th}JR(\th)$ we obtain the function
  \begin{align}
    \hat{F}_t
  \equiv \frac{\pt \hat{F}}{\pt t} = -P\cdot\dot{R}(\th)\hat{z} 
  = -\hat{Z}\cdot R(\th)^T\dot{R}(\th)\hat{z} 
  = -\dot{\th}\left(\hat{Z}\cdot R(\th)^TJR(\th)\hat{z}\right) 
  = \dot{\th}(\hat{z}\cdot J\hat{Z}).
  \end{align}
  
  Because $\th=\th(t)$ depends on $t$, by adding the function
  $\frac{\pt \hat{F}}{\pt t}$ to the Hamiltonian function $H$ in \eqref{HaFun}, as in Line 5 in p.272 of \cite{MS}, we obtain the Hamiltonian function $\hat{H}$ in the
  new coordinates:
  \be
  \hat{H}(t,\hat{Z},\hat{z}) = H_0(P,q) + \hat{F}_t  
  = \frac{1}{2}|\hat{Z}|^2 + (\hat{z}\cdot J\hat{Z})\dot{\th}
               - \sum_{i=1}^{N-1}{m_i\over|r(t)a_i-\hat{z}|},   \lb{P-07}
  \ee
  where the variables of $H_0$ are functions of $\th$, $\hat{Z}$, $\hat{z}$,  given by \eqref{P-03}.

  {\bf Step 2.} {\it Dilating coordinates via the polar radius $r=|z(t)|$.}
  
  We change the coordinates $\hat{\xi}$ to
  $\td{\xi}=(\td{Z},\td{z})$
  which dilate with $r=|z(t)|$ given by \eqref{rTh}. The position coordinates are transformed by $\hat{z} = r\td{z}$.
  It is natural to scale the momenta by $1/r$ to get $\hat{Z}=\td{Z}/r$. 
  But it
  turns out that the new transformation
  \be  \hat{Z}=\frac{1}{r}\td{Z}+\dot{r}\td{z}  \lb{P-09}\ee
  makes the resulting Hamiltonian function simpler. 
  This transformation is generated by the function
  \be  \td{F}(t, \td{Z}, \hat{z})
    = \frac{1}{r}\td{Z}\cdot\hat{z} + \frac{\dot{r}}{2r}|\hat{z}|^2,
                          \lb{P-10}\ee
  and is given by
  \be
  \td{z} = \frac{\pt \td{F}}{\pt \td{Z}} = \frac{1}{r}\hat{z},   
  \hat{Z} = \frac{\pt \td{F}}{\pt \hat{z}} = \frac{1}{r}\td{Z}+\frac{\dot{r}}{r}\hat{z} = \frac{1}{r}\td{Z}+\dot{r}\td{z}, \nn\ee
  with
  \be  \frac{\pt \td{F}}{\pt t}
  = -\frac{\dot{r}}{r^2}\td{Z}\cdot\hat{z}
         + \frac{\ddot{r}r-\dot{r}^2}{2r^2}|\hat{z}|^2
  = -\frac{\dot{r}}{r}\td{Z}\cdot\td{z}
         + \frac{\ddot{r}r-\dot{r}^2}{2}|\td{z}|^2,     \lb{P-11}\ee
  by \eqref{P-09}.
  
  In this case, as in the last two lines on p.272 of \cite{MS}, the
  Hamiltonian function $\hat{H}$ in \eqref{P-07}
  becomes the new Hamiltonian function $\td{H}$ in the new coordinates:
  \bea
  \td{H}(t,\td{Z},\td{z})
  &\equiv& \hat{H}(t, \hat{Z}, \hat{z}) + \td{F}_t \nn\\
  &=& \frac{1}{2r^2}|\td{Z}|^2 + \frac{\dot{r}}{r}\td{Z}\cdot\td{z}
     + \frac{\dot{r}^2}{2}(|\td{z}|^2) + (\td{z}\cdot J\td{Z})\dot{\th}
     - \sum_{i=1}^{N-1}{m_i\over|r(t)a_i-r(t)\tilde{z}|} + \tilde{F}_t   \nn\\
  &=& \frac{1}{2r^2}|\td{Z}|^2+ \frac{r\ddot{r}}{2}|\td{z}|^2
     +(\td{z}\cdot J\td{Z})\dot{\th}
     - \frac{1}{r}\sum_{i=1}^{N-1}{m_i\over|a_i-\tilde{z}|}.   \lb{P-12}\eea

  {\bf Step 3.} {\it Coordinates via the true anomaly $\th$ as the independent variable.}
  
  Here we use the true anomaly $\th\in [0,2\pi]$ as an independent variable instead of $t\in [0,T]$ to simplify
  analysis. This is achieved by dividing the Hamiltonian function $\td{H}$ in \eqref{P-12} by $\dot{\th}$. Assuming
  $\dot{\th}(t)>0$ for all $t\in [0,T]$ and  $\td{\xi}\in W^{1,2}(\R/(T\Z), \R^8)$ we consider the action functional
  corresponding to the Hamiltonian system:
  \bea f(\td{\xi})
  &=& \int_0^T(\frac{1}{2}\dot{\td{\xi}}(t)\cdot J\td{\xi}(t) - \td{H}(t,\td{\xi}(t))) dt  \nn\\
  &=& \int_0^{2\pi}\left(\frac{1}{2}\frac{\dot{\td{\xi}}(t(\th))}{\dot{\th}(t)}\cdot J\td{\xi}(t)
            - \frac{\td{H}(t,\td{\xi}(t(\th)))}{\dot{\th}(t)}\right) d\th  \nn\\
  &=& \int_0^{2\pi}\left(\frac{1}{2}\td{\xi}'(\th)\cdot J\td{\xi}(\th) - \td{H}(\th,\td{\xi}(\th))\right) d\th.  \nn\eea
  Here we used $\td{\xi}'(\th)$ to denote the derivative of $\td{\xi}(\th)$ with respect to $\th$. But
  in the following we shall still write $\dot{\td{\xi}}(\th)$ for the derivative with respect to $\th$ instead of
  $\td{\xi}'(\th)$ for notation simplicity.
  
  Note that the elliptic Kepler orbit \eqref{rTh} satisfies
  $ r(t)^2\dot{\th}(t) = \sqrt{\mu p} = \sqrt{\mu a(1-e^2)} = \sg^2$ with $\sg=(\mu p)^{1/4}.$
  Note that $a=\mu^{1/3}(T/2\pi)^{2/3}$ with $T$ being the minimal period of the orbit \eqref{rTh}, we have
  $$  \sg = (\mu a(1-e^2))^{1/4} = \mu^{1/3}(\frac{T}{2\pi})^{1/6}(1-e^2)^{1/4}\;\in\; (0,\mu^{1/3}(\frac{T}{2\pi})^{1/6}] $$
  depending on $e$, when the mass $\mu$ and the period $T$ are fixed. Note that similarly we have $p=\sg^4/\mu$ depends
  on $e$ too. Note that the function $r$ satisfies
  $$  \ddot{r} = \frac{\mu p}{r^3} - \frac{\mu}{r^2} = \mu\left(\frac{p}{r^3} - \frac{1}{r^2}\right).  $$
  
  Therefore we get the Hamiltonian function $\td{H}$ in the new coordinates:
  \bea
  \td{H}(\th,\td{Z},\td{z})
       &\equiv& \frac{1}{\dot{\th}}\td{H}(t,\td{Z},\td{z})   \nn\\
  &=& \frac{1}{2r^2(t)\dot{\th}(t)}|\td{Z}|^2
         + \frac{r(t)\ddot{r}(t)}{2\dot{\th}(t)}|\td{z}|^2 
         + \td{z}\cdot J\td{Z}
         - \frac{1}{r(t)\dot{\th}(t)}\sum_{i=1}^{N-1}{m_i\over|a_i-\tilde{z}|}  \nn\\
  &=& \frac{1}{2\sg^2}|\td{Z}|^2+ \td{z}\cdot J\td{Z}
         + \frac{\mu(p-r(\th))}{2\sg^2}|\td{z}|^2
         - \frac{r(\th)}{\sg^2}\sum_{i=1}^{N-1}{m_i\over|a_i-\tilde{z}|},  \lb{P-13}\eea
  where $r(\th)=p/(1+e\cos\th)$. Note that now the minimal period $T$ of the elliptic solution $\td{z}=\td{z}(\th)$
  becomes $2\pi$ in the new coordinates in terms of true anomaly $\th$ as an independent variable.

  {\bf Step 4.} {\it Coordinates via the dilation of $\sg=(p\mu)^{1/4}$.}
  
  The last transformation is the dilation $(\td{Z},\td{z}) \;=\;
        (\sg\bar{Z},\sg^{-1}\bar{z})$.
  This transformation is symplectic and independent of the true anomaly $\th$. Thus the
  Hamiltonian function $\td{H}$ in \eqref{P-13} becomes a new Hamiltonian function:
  \be
  H(\th,\bar{Z},\bar{z}) \equiv
      \td{H}(\th,\sg\bar{Z},\sg^{-1}\bar{z})  
      = \frac{1}{2}|\bar{Z}|^2 + \bar{z}\cdot J\bar{Z}
        + \frac{p-r}{2p}|\bar{z}|^2
        - \frac{r}{\sg}\sum_{i=1}^{N-1}{m_i\over|\sg a_i-\bar{z}|}.  \lb{P-15}
  \ee
  The proof is complete. 
\end{proof}

Suppose that  $(P(t),q(t))^T$ is the solution to  system \eqref{1.2} with
$
q(t)=r(t)R(\theta(t))a_{N}$, and $P(t)=\dot{q}(t).$
By Proposition  \ref{P2.1}, it is transformed to the new solution $(\bar{Z}(\theta),\bar{z}(\theta))^T$ in the true anomaly $\theta$ as the new
variable for the original Hamiltonian function $H$ of \eqref{HaFun}, which is
given by
\be
\bar{Z}(\theta)
=(0,  \sigma)^T
\qquad
\bar{z}(\theta) = 
(\sigma, 0)^T
\lb{sigma-solution}\ee
Therefore, we can prove the Theorem \ref{linearized.Hamiltonian} directly. 

\begin{proof}[Proof of Theorem \ref{linearized.Hamiltonian}]
       In this proof we omit all the upper bars on the variables of $H$ in \eqref{eqn:red.Ham}. By \eqref{eqn:red.Ham}, we have
       \begin{align}
        H_z=JZ+\frac{p-r}{p}z
        -\frac{r}{\sigma}{\partial\over\partial z}\left(\sum_{i=1}^{N-1}{m_i\over|\sg a_i-\bar{z}|}\right)|_{z=\sg a_{N}},  \quad H_{Z} =Z-Jz, 
       \end{align}
       and
       \begin{align}
        H_{ZZ}=I,\quad H_{Zz}=-J, \quad H_{zZ}=J, \quad 
        H_{zz}=\frac{p-r}{p}I-\frac{r}{\sigma}{\partial^2\over\partial z^2}\left(\sum_{i=1}^{N-1}{m_i\over|\sg a_i-\bar{z}|}\right)|_{z=\sg a_{N}},
       \end{align}
       where all the items above are $2\times2$ matrices, and we denote by $H_x$ and $H_{xy}$
       the derivative of $H$ with respect to $x$, and the second derivative of $H$ with respect to
       $x$ and then $y$ respectively for $x$ and $y\in\R$.
       
       Now evaluating the corresponding functions at the special solution
       $({0,\sg},{\sg,0})^T\in\R^{4}$
       of \eqref{sigma-solution} with $z=(\sg,0)^T$, and summing them up, we
       obtain
       \begin{eqnarray}
       \frac{\partial^2}{\partial z^2}\left(\sum_{i=1}^{N-1}{m_i\over|\sg a_i-\bar{z}|}\right)\Bigg|_{z=\sigma a_N}
       &=&\left[\left(-\sum_{i=1}^{N-1}{m_i\over|\sigma a_i-z|^3}\right)I
          +3\sum_{i=1}^{N-1} m_i\frac{(\sg a_i-z)(\sg a_i-z)^T}{|\sigma a_i-z|^5}\right]\Bigg|_{z=\sigma a_N}
       \nonumber\\
       &=&-{1\over\sg^3}\left(\sum_{i=1}^{N-1}{m_i\over|a_i-a_{N}|^3}\right)I
       +{3\over\sg^3}\sum_{i=1}^{N-1} m_i\frac{(a_i-a_{N})(a_i-a_{N})^T}{|a_i-a_{N}|^5}.
       \label{U_zz}
       \end{eqnarray}
       Hence, we have
       \begin{eqnarray}
       H_{zz}&=&I_2-{r\over p}I_2-{r\over p}{\sg^3\over\mu}
                  \left[-{1\over\sg^3}\left(\sum_{i=1}^{N-1}{m_i\over|a_i-a_{N}|^3}\right)I
                  +{3\over\sg^3}\sum_{i=1}^{N-1} m_i\frac{(a_i-a_{N})(a_i-a_{N})^T}{|a_i-a_{N}|^5}\right]
                  \nn\\
             &=&I_2-{r\over p}\left[I_2-{1\over\mu}\left(\sum_{i=1}^{N-1}{m_i\over|a_i-a_{N}|^3}\right)I
             +{3\over\mu}\sum_{i=1}^{N-1} m_i\frac{(a_i-a_{N})(a_i-a_{N})^T}{|a_i-a_{N}|^5}\right],
       \end{eqnarray}
       where the last equality of the first formula follows from the definition of $\mu$ in \eqref{eq.of.cc}.
       Thus the proof is complete. 
\end{proof}

\setcounter{equation}{0}
\section{When the primaries form an Euler-Moulton central configuration}\label{sec:3}

We now consider the linear stability of special relative equilibrium in the $N$ body problem with one small mass away from the line of the $N-1$ primaries which form an Euler-Moulton central configuration. 
A typical example is the ERE orbit of the restricted
$4$-bodies,  involving  the Sun, the Earth, the Moon and one space station. Based on the analytical results, we numerically compute a special case where \( m_1 = m_3 \) for \( N = 4 \).
 in Section 3.2.

\subsection{Analytical results on the linear stability}
\begin{proposition}\label{prop.sum.lm3.lm4}
  Suppose that $a_i$ for $1\leq i \leq N-1$ forms an Euler collinear configuration and $a_i$ for $1 \leq i \leq N$ span $\R^2$. The eigenvalues $\lm_3$ and $\lm_4$ of $D$ defined by \eqref{matrix.D} are both positive and satisfy
  \begin{align}
    \lm_3 + \lm_4 = 3.
  \end{align}
\end{proposition}
\begin{proof}
By \eqref{matrix.D} and direct computations, we have that 
\begin{align}\lb{eqn:sum.l3.l4}
  \lambda_3+\lambda_4 =tr(D)
  =2+{1\over\mu}\sum_{i=1}^{N-1}{m_i\over|a_i-a_N|^3} 
  =3+\beta_{2,0},
  \end{align}
where $\beta_{2,0}={1\over\mu}\sum_{i=1}^{N-1}{m_i\over|a_i-a_N|^3}-1.$
Moreover, we have
\begin{align}
  D-{3+\beta_{2,0}\over2}I_2
&=-{3\over2\mu}\sum_{i=1}^{N-1}{m_i\over|a_i-a_N|^3}I_2
 +{3\over\mu}\sum_{i=1}^{N-1}{m_i(a_i-a_N)(a_i-a_N)^T\over|a_i-a_N|^5}
 \\
&={3\over2\mu}\left[\sum_{i=1}^{N-1}m_i{2(a_i-a_N)(a_i-a_N)^T-|a_i-a_N|^2\over|a_i-a_N|^5}\right]\\
&=\psi(\beta_{22,0}),
\end{align}
where
$\beta_{22,0}={3\over2\mu}\sum_{i=1}^{N-1}m_i{(z_{a_i}-z_{a_N})^2\over|a_i-a_N|^5}$
and $\psi(z) = (\begin{smallmatrix}
  x & y\\-y & x
\end{smallmatrix})$ with $z = x+ \sqrt{-1} y$.
Therefore, following direct computation,
we have
\begin{align}\lb{eqn:l3l4}
  \lambda_3,\lambda_4={3+\beta_{22,0}\over2}\pm|\beta_{22,0}|.
\end{align}

Now suppose $a_1,a_2,\ldots,a_{N-1}$ form an Euler-Moulton collinear configuration located on the $x$-axis in $\R^2$. 
We set $a_i=(a_{ix},a_{iy})$ for $i=1,2,\ldots,N$,
and hence $a_{iy}=0$ for $i=1,2,\ldots,N-1$.
Then the central configuration equation \eqref{eq.of.cc} of $m_N$ gives
$
\sum_{j=1}^{N-1}{m_j(a_j-a_N)\over|a_j-a_N|^3}=-\mu a_N.
$
Especially, for $a_{Ny}$, we have that 
\begin{equation}
\sum_{j=1}^{N-1}{m_ja_{Ny}\over|a_j-a_N|^3}=\mu a_{Ny}.
\end{equation}
Since $a_{Ny}\ne0$, it follows that 
\begin{equation}\label{second.formula.of.mu}
\mu=\sum_{j=1}^{N-1}{m_j\over|a_j-a_N|^3},
\end{equation}
and hence $\beta_{2,0}=0$. By \eqref{eqn:l3l4}, we have
$\lambda_3+\lambda_4=3$
and
\begin{align}
D={3\over\mu}\sum_{j=1}^{N-1}{m_j(a_j-a_N)(a_j-a_N)^T\over|a_j-a_N|^5}
={3\over\mu}
\begin{pmatrix}
       \sum_{i=1}^{N-1}{m_ix_{Ni}^2\over r_{Ni}^5} &
	\sum_{i=1}^{N-1}{m_ix_{Ni}y_{Ni}\over r_{Ni}^5}\\
	\sum_{i=1}^{N-1}{m_ix_{Ni}y_{Ni}\over r_{Ni}^5} &
	\sum_{i=1}^{N-1}{m_iy_{Ni}^2\over r_{Ni}^5} 
\end{pmatrix},\lb{eqn:D}
\end{align}
where $(x_{Ni},y_{Ni})=a_N-a_i$ and $r_{Ni}=|a_N-a_i|$.
Using Cauchy's inequality, we have
\begin{equation}
\left(\sum_{i=1}^{N-1}{m_ix_{Ni}y_{Ni}\over r_{Ni}^5}\right)^2
=\left(\sum_{i=1}^{N-1}\sqrt{m_ix_{Ni}^2\over r_{Ni}^5}\sqrt{m_iy_{Ni}^2\over r_{Ni}^5}\right)^2
\le\left(\sum_{i=1}^{N-1}{m_ix_{Ni}^2\over r_{Ni}^5}\right)
   \left(\sum_{i=1}^{N-1}{m_iy_{Ni}^2\over r_{Ni}^5}\right).
\end{equation}
Then we have $\det(D)\ge0$ and hence $\lambda_3,\lambda_4 \geq 0$. 
\end{proof}

\begin{remark}
Note that for $N=4$,  $\beta_{2,0}$ and $\beta_{22,0}$ coincide with
those of (2.10) and (A.14) in \cite{LiZ}.
But the methods used in \cite{LiZ} can not be used
here, even for $N=4$.
\end{remark}

Without loss of the generality, we may assume that $\lambda_3 \geq \lambda_4$ in the following.
Since $\lambda_3 + \lambda_4 = 3$ and $\lambda_3, \lambda_4 > 0$, we can denote $\lambda_3$ and $\lambda_4$ by variable $\beta$ as follows: 
\begin{align}
  \lambda_3 = \frac{3 + \sqrt{9 -\beta}}{2}, \quad \lambda_4 = \frac{3 - \sqrt{9 -\beta}}{2}.
\end{align}

According to (2.17) and (2.18) of \cite{HLS}, the essential part $\gamma=\gamma_{\beta, e}(t)$ of the fundamental solution of the elliptic Lagrangian solution satisfies
\begin{align}
\begin{cases}
&\dot{\gamma}(t)=J B(t) \gamma(t), \\
&\gamma(0)=I_{4},
\end{cases}\lb{eqn:lagga}
\end{align}
with
$B(t)=\left(\begin{smallmatrix}
I & -J \\
J & I_2 - \frac{K_{\beta}}{1+e\cos t}
\end{smallmatrix}\right),
$
where $e$ is the eccentricity and $K_{\beta}=\diag\{\frac{3+\sqrt{9-\beta} }{2}, \frac{3-\sqrt{9-\beta}}{2}\}$.
 For $(\beta, e) \in[0,9) \times[0,1)$, the second order differential operator corresponding to (\ref{eqn:lagga}) is given by
\begin{align}\lb{eqn:ca}
  \cA(\bb, e) &= -\frac{\d^2}{\d t^2}I_2  -I_2 + \frac{1}{1+e\cos t}R(t)K_{\bb}R(t)^T \nn \\
  &= -\frac{\d^2}{\d t^2}I_2  -I_2 + \frac{1}{2(1+e\cos t)}(
  (\lambda_3 + \lambda_4)I_2 + (\lambda_3 - \lambda_4) S(t))\\
  &= -\frac{\d^2}{\d t^2}I_2  -I_2 + \frac{1}{2(1+e\cos t)}(
       3I_2 + \sqrt{9 - \beta} S(t)),  \lb{cA}
\end{align}
where $R(t) = (\begin{smallmatrix}\cos t & -\sin t\\ \sin t & \cos t \end{smallmatrix})$,
$S(t) = (\begin{smallmatrix}
\cos 2t & \sin 2t\\ \sin 2t & -\cos 2t
\end{smallmatrix})$, defined on the domain
$$
\ol{D}(\om,2\pi)=\{y\in W^{2,2}([0,2\pi],\C^n)\;
|\;y(2\pi)=\om y(0),\dot{y}(2\pi)=\om\dot{y}(0)\}.
$$ 
Then it is self-adjoint and depends on the parameters $\beta$ and $e$. 
Therefore, the linear stability of EREs of such a  restricted 4-body problem can be directly derived from the linear stability of the elliptic Lagrangian solutions in \cite{HLS} in terms of $\beta$ and $e$.
Consequently, Theorem \ref{thm:RE.norm.form}
immediately follows from Theorem 1.2 of \cite{HLS}.
Readers may refer to \cite{HLS} for detailed discussions. 




\subsection{Numerical results on the linear stability when 
in the symmetric case of four-body problem}
\label{subsec:3.2}
The dependence between the mass parameter $\beta$ and the masses $m=(m_1,m_2,\ldots,m_{N-1})$
is quite intricate.
In this section, to shed light on how the linear stability varies with specific mass distributions, we will conduct some numerical computations.

It is well-known that the linear stability of an elliptic Euler solution
of the $3$-body problem with masses $m=(m_1,m_2,m_3)\in (\R^+)^3$ is determined by the
eccentricity $e\in [0,1)$ and the mass parameter
\begin{align}\lb{1.4}
  \bb = \frac{m_1(3x^2+3x+1)+m_3x^2(x^2+3x+3)}{x^2+m_2[(x+1)^2(x^2+1)-x^2]},
\end{align}
where $x$ is the unique positive solution of the Euler quintic polynomial equation
\be (m_3+m_2)x^5+(3m_3+2m_2)x^4+(3m_3+m_2)x^3-(3m_1+m_2)x^2-(3m_1+2m_2)x-(m_1+m_2)=0, \lb{Euler.quintic.polynomial}
\ee
and the three bodies form a central configuration of $m$, which are denoted by $q_1=0$,
$q_2=(x\alpha,0)^T$ and $q_3=((1+x)\alpha,0)^T$ with $\alpha=|q_2-q_3|>0$ and $x\alpha=|q_1-q_2|$ (\cite{Zhou2017}).

Assuming further that $m_1 = m_3$, we have that $x=1$ is the unique positive root of \eqref{Euler.quintic.polynomial}.
Given the center of mass of the three primaries is the origin, the position $a_i$ can be expressed in terms of in $m_2$ as follows.
\begin{align}
a_1 = (-(1-m_2)^{-{1\over2}},0)^T,\quad 
a_2 = (0,0)^T,\quad 
a_3 = ((1-m_2)^{-{1\over2}},0)^T.
\end{align}
Since the center of mass is the origin,
we suppose $a_4=(0,y(1-m_2)^{-{1\over2}})^T$ for some $y\in\R^+$.
By \eqref{second.formula.of.mu}, we have
\begin{equation}\label{y}
{1-m_2\over(y^2+1)^{3\over2}} + {m_2\over y^3}={1+7m_2\over8}.
\end{equation}
Note that when $m_2=0$, \eqref{y} gives $y=\sqrt{3}$;
when $m_2=1$, \eqref{y} gives $y=1$.
If $0<y<1$, we have
\begin{equation}
{1-m_2\over(y^2+1)^{3\over2}} + {m_2\over y^3}\ge{1-m_2\over\sqrt{8}}+m_2\ge{1-m_2\over8}+m_2={1+7m_2\over8},
\end{equation}
where the equality holds if and only if $m_2=1$.
Hence we must have $y\ge1$.
Moreover, we take the derivative of $y$ and obtain that 
\begin{equation}\label{y.prime}
{\d y\over \d m_2}
=-\frac{{7\over8}+{1\over(y^2+1)^{3\over2}}-{1\over y^3}}
       {3y\Big[{1-m_2\over(y^2+1)^{5\over2}} + {m_2\over y^5}\Big]}
       <0,
\end{equation}
where $m_2\in[0,1)$.
Here we use
$$
{1\over y^3}-{1\over(y^2+1)^{3\over2}}<{1\over y^3}-{1\over (y+1)^3}
={3y^2+3y+1\over y^3(y+1)^3}={3\over y^2(y+1)^2}+{1\over y^3(y+1)^3}
<{3\over4}+{1\over8}={7\over8}.
$$
Therefore, the range of $y$ is $[1,\sqrt{3}]$,
and $y$ is strictly decreasing with respect to $m_2\in[0,1]$.
By direct computations,  we can reduce $D$, defined in \eqref{eqn:D}, to 
\begin{equation}
D=3\left(\begin{matrix}
{8(1-m_2)\over(1+7m_2)(y^2+1)^{5\over2}} & 0\\
0 & 1 - {8(1-m_2)\over(1+7m_2)(y^2+1)^{5\over2}}
\end{matrix}\right).
\end{equation}
Let
\begin{equation}\label{z}
z = {8(1-m_2)\over(1+7m_2)(y^2+1)^{5\over2}}.
\end{equation}
We then have $\lambda_3=3(1-z)$, $\lambda_4=3z$ and $\det D=9z(1-z)$. It follows that 
\begin{equation}\label{beta}
\beta = 9-(\lambda_3-\lambda_4)^2=36z(1-z).
\end{equation}

If $m_2=0$, then $y=\sqrt{3}$, hence \eqref{z} gives $z={1\over4}$,
and \eqref{beta} gives $\beta=9$;
if $m_2=1$, \eqref{z} gives $z=0$,
and \eqref{beta} gives $\beta=0$.
According to Theorem \ref{thm:RE.norm.form}
(also refer to Theorem 1.2 of \cite{HLS}), 
if $0\le\beta<1$ and $e=0$, the ERE is linearly stable.
We now have the relation of $\beta$ and $m_2$, and can determine  the stability of the ERE in terms of $m_2$,  and then
Theorem \ref{thm:numerically} holds.
Furthermore, by numerical computations,
we have the critical value $m^*\approx0.854$,
see Figure \ref{fig:1}(b).

\setcounter{equation}{0}
\section{When the primaries form a regular polygon central configuration}\label{sec:4}

We now consider the linear stability of special relative equilibrium in the $N-$ body problem with one small mass when the $N-1$ primaries form a $(1+n)$-gon central configuration. 

\subsection{Analytical results on the linear stability}

Suppose that $a_i$ for $1\leq i\leq N-1$ form a  $(1+n)$-gon configuration. 
By Proposition \ref{polygon.prop.lm3.lm4},
the eigenvalues $\lm_3$ and $\lm_4$ of $D$ defined by \eqref{matrix.D} satisfy
\begin{align}
	\lm_3 = 1+{A\over\omega^2}+{B\over\omega^2},\quad
	\lm_4 = 1+{A\over\omega^2}-{B\over\omega^2}={l_3(x,\theta)\over\omega^2},
\end{align}
where  $\omega^2,A,B$ and $l_3(x,\theta)$ are given by p.305 and p.311 of \cite{BaE}.

According to (2.17) and (2.18) of \cite{HLS}, the essential part $\gamma=\gamma_{m, e}(t)$ of the fundamental solution of the elliptic Lagrangian solution satisfies
\begin{equation}\label{eqn:ga.n-gon}
\begin{cases}
&\dot{\gamma}(t)=J B(t) \gamma(t), \\
&\gamma(0)=I_{4},
\end{cases}
\end{equation}
with
$B(t)=\left(\begin{smallmatrix}
I & -J \\
J & I_2 - \frac{K_{m}}{1+e\cos t}
\end{smallmatrix}\right),
$
where $e$ is the eccentricity and $K_{m}=\diag\{\lambda_3, \lambda_4\}$, the subscript "$m$" denote the dependence with respect to the masses.
 For $(\beta, e) \in[0,9) \times[0,1)$, the second order differential operator corresponding to \eqref{eqn:ga.n-gon} is given by
\begin{align}\lb{eqn:ca.n-gon}
  \cA(m, e) &= -\frac{\d^2}{\d t^2}I_2  -I_2 + \frac{1}{1+e\cos t}R(t)K_mR(t)^T \nn \\
  &= -\frac{\d^2}{\d t^2}I_2  -I_2 + \frac{1}{1+e\cos t}\left[
  \frac{\lambda_3 + \lambda_4}{2}I_2 
  + \frac{\lambda_3 - \lambda_4}{2} S(t)\right], 
\end{align}
defined on the domain $\ol{D}(\om,2\pi)$.
Then it is self-adjoint and depends on the parameters $m$ and $e$. 

We define the $\omega$-Morse index $\phi_{\omega}(\mathcal{A}(m,e))$ to be the total number of negative eigenvalues of $\mathcal{A}(m,e)$,
and define $\nu_{\omega}(\mathcal{A}(m,e))=\dim\ker(\mathcal{A}(m,e))$.
In the Appendix \ref{subsec:5.2}, we provide a brief review of the Maslov-type
$\omega$-index theory for $\omega$ on the unit circle in the complex plane,  following \cite{Lon4}.
In the following, we use notations introduced there.


The next theorem follows from the corresponding property
of the Maslov-type index.

\begin{theorem} \label{criteria of spectral stability}
(See (9.3.3) on p.204 of Long \cite{Lon4} with $\omega_0=-1$  and $k=2$)
The matrix $\gamma(2\pi)$ is spectral stable,
if $|\phi_{1}(\mathcal{A}(m,e))-\phi_{-1}(\mathcal{A}(m,e))|=2$.
The matrix $\gamma(2\pi)$ is hyperbolic,
if $\mathcal{A}(m,e)$ is positive definite
in $\overline{D}(\omega,2\pi)$ for any $\omega\in\U$.

\end{theorem}

Note that, if
$\alpha=\frac{\lambda_3+\lambda_4}{2}-1={A\over\om^2}$ and
$\beta=\frac{\lambda_3-\lambda_4}{2}={|B|\over\om^2}$, $\cA(m,e)$ is the same as the operator
\begin{equation}\label{bar.A}
    \bar\cA(\alpha,\beta,e)= -\frac{\d^2}{\d t^2}I_2  -I_2 +
    r_e(t)[(1+\alpha)I_2+\beta R(t)NR(t)^T],
\end{equation}
defined on p.~1028
of \cite{Hu2020},
where $r_e(t)=\frac{1}{1+e\cos t},R(t)NR(t)^T=S(t)$.   
Hence by Lemma \ref{Lm:A.12},
$$
\phi_{\omega}(\mathcal{A}(m,e))
=\phi_{\omega}(\bar{\mathcal{A}}(\alpha,\beta,e)),\quad
\nu_{\omega}(\mathcal{A}(m,e))
=\nu_{\omega}(\bar{\mathcal{A}}(\alpha,\beta,e)).
$$

\subsection{Linear instability of \texorpdfstring{$S_1$}{S\_1} and \texorpdfstring{$S_2$}{S\_2}}
By Proposition \ref{the eigenvalues of the case S_1 and S_2},
when ${m_0\over M}\to+\infty$,
we have $\alpha\to2$ and $\beta\to6$;
and hence, we have
\begin{eqnarray}
    \lim_{{m_0\over M}\to+\infty}\bar\cA(\alpha,\beta,e)=\bar\cA(2,6,e)= -\frac{\d^2}{\d t^2}I_2  -I_2 +
    r_e(t)[3I_2+6 R(t)NR(t)^T].
\end{eqnarray}

Now we can give:

\begin{proof}[Proof of Theorem \ref{no stability of S_1 and s_2 }]
Note that $\bar\cA(2,6,e)$
coincides with the following operator (see (2.42) on p.1263 of \cite{Zhou2017})
\begin{equation}
    A(\beta,e) = -\frac{\d^2}{\d t^2}I_2  -I_2 +
   {1\over2(1+e\cos{t})}[(3+\beta)I_2+2(1+\beta)R(t)NR(t)^T],
\end{equation}
with $\beta=3$. Furthermore, by Theorem 1.3 and Theorem 1.5(ix) of \cite{Zhou2017}, we can always choose a fixed $e_0$ such that for any $e\leq e_0$, $\gamma(2\pi)\approx R(\theta)\diamond D(2)$ when $\hat{\beta_2}<\beta=3<\hat{\beta_{5\over 2}}$ , that is, $S_1$ and $S_2$ is unstable when ${m_o\over M}$ is sufficiently large and $e\in [0,e_0]$. 
\end{proof}

\subsection{Linear stability of \texorpdfstring{$S_3$}{S\_3}}

\begin{lemma}
    For any $(m_0,e)\in(0,+\infty)\times[0,1)$,
    $\cA(m,e)$ is positive in $\ol{D}(1,2\pi)$.
\end{lemma}

In fact, from Remark 3.9 of \cite{Hu2020}, we have
\begin{equation}
    \phi_{1}(\bar\cA({1\over2},\beta,e))=0,
    \quad
    \nu_{1}(\bar\cA({1\over2},{3\over2},e))=0,
    \quad
    \forall (\beta,e)\in[0,3/2)\times[0,1),
     \end{equation}
    \begin{equation}
      \phi_{-1}(\bar\cA({1\over2},{3\over2},e))=2,
    \quad
    \nu_{-1}(\bar\cA({1\over2},{3\over2},e))=0,
    \quad
    \forall \beta\in (\beta_m(e),{3\over2}), \, e\in[0,1),
\end{equation}
here $\beta_m(e)$ is an analytic curve in $e$. 

By Proposition \ref{A.3},
we have $\alpha>{1\over2}$ and
$\alpha+1>\beta$;
and hence, if $\beta<{3\over2}$,
we have
$\bar\cA(\alpha,\beta,e)>\bar\cA({1\over2},\beta,e)>0$ in $\ol{D}(1,2\pi)$.
Moreover, if $\beta\ge{3\over2}$,
we can deduce the following formula:
\begin{align}
    \bar \cA(\alpha,\beta,e)&=-{d^2\over dt^2}I_2-I_2+{1\over1+e\cos{t}}\big[(1+\alpha)I_2+\beta S(t)\big]\\
    &=-{d^2\over dt^2}I_2-I_2+{1\over1+e\cos{t}}\big[{3\over 2}I_2+{3\over 2}S(t)\big]\\
    &\quad +{1\over1+e\cos{t}}\big[(\alpha-{1\over2})I_2+(\beta -{3 \over 2})S(t)\big]\\
    &=\bar\cA({1\over2},{3\over2},e)+{1\over1+e\cos{t}}(\beta-{3\over 2})(I_2+S(t))+{1\over1+e\cos{t}}(\alpha-\beta +1)I_2\\
    &\quad>0,
\end{align}
when $\beta\ge{3\over 2}$ and $\alpha-\beta+1>0.$ 
Hence we have
\begin{lemma}
    (i) We have
    \begin{equation}
    \bar\cA(\alpha,\beta,e)>0
    \;\;
    {\rm in}\;\ol{D}(1,2\pi),
    \quad
    \forall {m_0\over M}\in(0,+\infty),e\in[0,1).
    \end{equation}
    
    (ii) There exists $m^*(e)$ depending on $e$ such that
    \begin{equation}
    \phi_{-1}(\cA(m,e))=2,
    \quad
    \nu_{-1}(\cA(m,e))=0,
    \quad
    \forall {m_0\over M}\in[m^*(e),+\infty),e\in[0,1).
    \end{equation}
\end{lemma}
Now from the above results, we can give 
:
\begin{proof}[Proof of Theorem \ref{stability of S_3}]
Under the assumption of the above lemma, for the fundamental solution $\gamma$ of \eqref{eqn:ga.n-gon}, we have
$$
  \phi_{-1}(\cA({m,e}))=2,\quad
  \nu_{-1}(\cA({m,e}))=0,\quad$$
  $$
  \phi_1(\cA({m,e}))=0,\quad
  \nu_1(\cA(m,e))=0,
$$
According to Theorem \ref{criteria of spectral stability} and the basic normal forms of symplectic matrices, $\gamma(2\pi)$ may have only two possible cases: $\gamma(2\pi) \approx R(\alpha) \diamond R (\beta)$,$\gamma(2\pi) \approx N_2(e^{-1\sqrt{-1}\theta},u)$, for $\alpha,\beta, \theta \in [0,2\pi) $and a real vector $u\in\R^4$. Moreover from (11) in p.207 0f \cite{Lon4} (also refer to the Appendix \ref{subsec:5.2}), we have the iteration formula as following:
$$
i_{-1}(\gamma)=i_1(\gamma)+S^+_{\gamma(2\pi)}(1)+\sum_{0<\theta<\pi}(S^+_{\gamma(2\pi)}(e^{\sqrt{-1}\theta})-S^-_{\gamma(2\pi)}(e^{\sqrt{-1}\theta}))-S^+_{\gamma(2\pi)}(-1)
$$
hence, we obtain
$$
2=\sum_{0<\theta<\pi}(S^+_{\gamma(2\pi)}(e^{\sqrt{-1}\theta})-S^-_{\gamma(2\pi)}(e^{\sqrt{-1}\theta}))
$$
By Lemma \ref{Lm:splitting.numbers},
we deduce $\gamma(2\pi) \approx R(\alpha) \diamond R (\beta)$, that is, $S_3$ is linearly stable.
\end{proof}

\appendix
\setcounter{equation}{0}
\section{Appendix}\label{sec:5}

\subsection{on 
    \texorpdfstring{$\lambda_3$}{lambda\_3} and 
    \texorpdfstring{$\lambda_4$}{lambda\_4}}

To establish the relations between Bang's quantities in \cite{BaE}
and ours, we use the notations below. Note that the first \(N-2\) primary bodies, each with mass \(m\), are placed at the vertices of a regular polygon and rotate rigidly around the \((N-1)\)-th mass \(m_0\) at the center, while the \(N\)-th body is massless, indicating that most of the mass is concentrated at the center.
We suppose
 \begin{center}
 	\begin{tabular}{c c}
 		the positions & the masses\\
 		$a_{1}=\alpha \omega_{1}$                        &$m_1=m$\\
 		$\vdots$                                       &$\vdots$\\
 		$a_{N-2}=\alpha \omega_{N-2}$               &$m_{N-2}=m$\\
 		$a_{N-1}=0$                                 &$m_{N-1}=m_0$\\
 		$a_{N}=\alpha w_{0}$                        &$m_{N}$=0,
 	\end{tabular}
 \end{center}
 where $\omega_i=e^{-\frac{2i\pi j}{N-2}}, j=1, \ldots , N-2$
 and $w_0$ being the affix of one of the points $S_1,S_2$
 or $S_3$.
Here the notations $m_0, m ,M=nm=(N-2)m$ and $w_0$ are used by
D.~Bang and B.Elmabsout in \cite{BaE}. In this case,
 the moment of inertia is
 \begin{displaymath}
 I=\alpha^{2}
 \sum_{i=1}^{N-2}m|\omega_i|^2=1,
 \end{displaymath}
 implying $\alpha =\frac{1}{\sqrt{(N-2)m}}=\frac{1}{\sqrt{M}}$, and 
 the normalization can be written as $$\sum_{i=1}^{N-2}m+m_0=1.$$
 
 Recall the following quantities defined by D.~Bang and B.~Elmabsout in\cite{BaE}:
$$h_n(1)=\frac{1}{4n}\sum_{j=1}^{n-1}\frac{1}{\sin(\frac{j\pi}{n})},$$
$$h_n(x,u)=\frac{1}{n}\sum_{j=1}^{n}\frac{1-x\thinspace cos(\frac{2j\pi}{n}+u)}{(1+x^2-2x\thinspace cos(\frac{2j\pi}{n}+u))^{3/2}}.$$
Here $n=N-2$.
The configuration guarantees the equilibrium equation
$$\alpha^3\omega^2=m_0+Mh_n(1),$$
and the additional equation (2) in \cite{BaE}:
$$\phi(\rho,\theta)=m_0(1-\rho^3)+M(h_n(\frac{1}{\rho},\theta)-h_n(1)),$$
where $M=nm=(N-2)m$.
Moreover,
 \begin{eqnarray}
 \omega^2&=&m_0+Mh_n(1)=m_0+\frac{M}{4n}\sum_{j=1}^{n-1}\frac{1}{\sin(\frac{j\pi}{n})},
 \\
 A&=&\frac{M}{2n}\sum_{j=1}^{n}\frac{1}{|w_0-\omega_j|^3}+\frac{m_0}{2|w_0|^3},
 \\
 B &=& \frac{3M}{2n}\sum_{j=1}^{n}\frac{1}{|w_0-\omega_j|^3}\frac{w_0-\om_j}{\ol{w_0-\om_j}}
 +\frac{3}{2}\frac{m_0}{|w_0|^3}\frac{w_0}{\ol{w_0}},
 \end{eqnarray}
 and
\begin{eqnarray}
l_2(x,\th)&=&\omega^2-A,
\\
l_3(x,\th)&=&\omega^2+A-|B|.
\end{eqnarray}
 
 \par 
 We denote 
 $$\psi(x+iy)=\left(\begin{array}{ll}
 x&y\\
 y&-x
 \end{array}
 \right).
 $$
 It is easy to see that 
 $$
 zz^{T}=\frac{1}{2}|z|^2I_{2}+\frac{1}{2}\psi(z^{2}),
 $$
 where $z=x+iy$ and $zz^{T}$ denotes the product of $2\times 1$ matrix and $1\times 2$ matrix.
 
 Using the notations and formulas above, we obtain

 \begin{proposition}\label{polygon.prop.lm3.lm4}
	Suppose that $a_i$ for $1\leq i \leq N-1$ forms a $(1+n)$-gon configuration. We have
    $$\mu\alpha^3=\omega^2.$$
    Moreover, the eigenvalues $\lm_3$ and $\lm_4$ of $D$ defined by \eqref{matrix.D} satisfy
	\begin{align}
	\lm_3 = 1+{A\over\omega^2}+{B\over\omega^2},\quad
	\lm_4 = 1+{A\over\omega^2}-{B\over\omega^2}={l_3(x,\theta)\over\omega^2}.
	\end{align}
\end{proposition}

 \begin{proof}
First, we have
 
 \begin{displaymath}
 tr(D)=\sum_{i=1}^{N-1}\frac{m_i}{|a_i-a_N|^3}  
 =\frac{1}{\alpha^3}\sum_{i=1}^{N-1}\frac{m}{|\omega_i-\omega_N|^3}+\frac{1}{\alpha^3}\frac{m_0}{|\omega_N|^3}
 =\frac{1}{\alpha^3}\cdot 2A,
 \end{displaymath}
 and
 \begin{eqnarray*}
 	&&
 	\sum_{i=1}^{N-1}m_i\frac{(a_i-a_N)(a_i-a_N)^T}{|a_i-a_N|^5}
 	\\
 	&=&
 	\sum_{i=1}^{N-2}m\frac{\frac{1}{2}\alpha^2|\omega_i-\omega_N|^2I_2+\frac{1}{2}\alpha^2\psi((\omega_i-\omega_N)^2)}
 	{\alpha^5|\omega_i-\omega_ N|^5}+
 	m_0\frac{\frac{1}{2}\alpha^2|\omega_N|^2+\frac{1}{2}\alpha^2\psi(\omega_N^2)}{\alpha^5|\omega_N|^5}
 	\\
 	&=&
 	\sum_{i=1}^{N-2}m(\frac{1}{2\alpha^3|\omega_i-\omega_N|^3}I_2+\frac{\psi((\omega_i-\omega_N)^2)}{2\alpha^3|\omega_i-\omega_N|^5})+
 	m_0(\frac{1}{2\alpha^3|\omega_N|^3}I_2+\frac{\psi((\omega_N)^2)}{2\alpha^3|\omega_N|^5})\\
 	&=&
 	[\sum_{i=1}^{N-2}\frac{m}{2\alpha^3|\omega_i-\omega_j|^3}I_2+\frac{m_0}{2\alpha^3|\omega|^3}I_2]+
 	[\sum_{i=1}^{N-2}m\frac{\psi((\omega_i-\omega_N)^2)}{2\alpha^3|\omega_i-\omega_N|^5}+\frac{\psi(\omega_N^2)}{2\alpha^3|\omega_N|^5}]\\
 	&=&
 	\frac{1}{\alpha^3}AI_2+\frac{1}{3\alpha^3}\psi(B).
 \end{eqnarray*}
 Therefore, we have the following formula
 \begin{eqnarray*}
 	D&=&I_2-\frac{1}{\mu}(\sum_{i=1}^{N-1}\frac{m_i}{|a_i-a_N|^3})I_2+\frac{3}{\mu}(\sum_{i=1}^{N-1}m_i\frac{(a_i-a_N)(a_i-a_N)^T}{|a_i-a_N|^5})\\
 	&=&I_2-\frac{1}{\mu\alpha^3}\cdot 2AI_2+\frac{3}{\mu\alpha^3}(AI_2+\frac{1}{3}\psi(B))\\
 	&=&I_2+\frac{1}{\mu\alpha^3}(AI_2+\psi(B))\\
 	&=&I_2+\frac{A}{\mu\alpha^3}I_2+\frac{1}{\mu\alpha^3}\psi(B)\\
 	&=&I_2+\frac{1+\beta_{2.0}}{2}I_2+\psi(\beta_{22.0}).
 \end{eqnarray*}
 This implies:
 \begin{eqnarray*}
 	\frac{1+\beta_{2.0}}{2}=\frac{A}{\mu\alpha^3},\quad 
 	\beta_{22.0}=\frac{B}{\mu\alpha^3}.
 \end{eqnarray*}
 Let $u=\frac{1+\beta_{2.0}}{2}$ and $ v=\beta_{22.0}$, hence we obtain
 \begin{displaymath}
 uI_2+\psi(v)=\left(\begin{array}{cc}
 u+v_x&v_y\\
 v_y&u-v_x
 \end{array}
 \right),
 \end{displaymath}
 where $v_x=\text{Re}\,\beta_{22.0}$ and $ v_y=\text{Im}\,\beta_{22.0}$, indicating
 $\text{tr}(uI_2+\psi(v))=2u$ and $\text{det}(uI_2+\psi(v))=u^2-|v|^2$. 
 Therefore, the characteristic polynomial  of $uI_2+\psi(v)$ is $\lambda^2-2u\lambda+u^2-|v|^2$,
 and then the eigenvalues of the matrix $uI_2+\psi(v)$ are $u\pm|v|$. 
 Thus the eigenvalues of D are
 \begin{eqnarray*}
 	\lambda_{3,4}=1+\frac{1+\beta_{2.0}}{2}\pm|\beta_{22.0}|=1+\frac{A}{\mu\alpha^3}\pm\frac{|B|}{\mu\alpha^3}.
 \end{eqnarray*}
 The following formulas are the direct results from above arguments.
 \begin{eqnarray*}
 	\mu&=&\sum_{1\leq i<j\leq N-1}\frac{m_im_j}{|a_i-a_j|}\\
 	&=&\sum_{1\leq i<j\leq N-2}\frac{m^2}{\alpha|\omega_i-\omega_j|}+\sum_{i=1}^{N-2}\frac{mm_0}{\alpha\omega_i},
 \end{eqnarray*}
 then
 \begin{eqnarray*}
 	\mu\alpha^3&=&\alpha^2(\sum_{1\leq i<j\leq N-2}\frac{m^2}{\omega_i-\omega_j}+\sum_{i=1}^{N-2}\frac{mm_0}{|\omega_i|})
 	\\&=&\frac{m}{N-2}\sum_{1\leq i<j\leq N-2 }\frac{1}{|\omega_i-\omega_j|}+m_0
 	\\&=&\frac{m}{N-2}\sum_{i<j}\frac{1}{|\omega_i-\omega_j|}+m_0
 	\\&=&\frac{m}{2}\sum_{1<j\leq N-2}\frac{1}{|\omega_0-\omega_j|}+m_0
 	\\&=&\omega ^2.
 \end{eqnarray*}
 Consequently, $\lambda_4=\frac{l_3(x,\theta)}{\mu\alpha^3}
 =\frac{l_3(x,\theta)}{\omega^2}$, where $l_3(x,\theta)=A+\omega^2-|B|$.
\end{proof}

For the central configurations when the massless body lies
at $S_1$ or $S_2$, we have
\begin{proposition}\label{the eigenvalues of the case S_1 and S_2}
If $\rho=S_1$ (or $S_2$), we have
$
l_3(\rho,0)<0$;
and hence,
$\lambda_4<0.$ Moreover, $\frac{A}{\omega^2}\to 2$ and $\frac{B}{\omega^2}\to6$, as  ${m_0\over M}\to +\infty$
\end{proposition}
\begin{proof}
    The first assertion immediately follows by Lemma 2 of \cite{BaE}
    and Proposition \ref{polygon.prop.lm3.lm4}.
    Let us verify the second assertion in detail. Through the equilibrium equation 
    $$
    \psi_1(\rho)=m_0(1-\rho^3)+M\left(h_n\left(\frac{1}{\rho}\right)-h_n(1)\rho^3\right)=0,
    $$
     defined on p.309 in \cite{BaE},
    we have the following formula:
    $$
    m_0=\frac{M(h_n(\frac{1}{\rho})-h_n(1)\rho^3)}{\rho^3-1}=\frac{M(s^3h_n(s)-h_n(1))}{1-s^3}
    $$
    $$
    \omega^2=m_0+Mh_n(1)=\frac{Ms^3}{1-s^3}(h_n(s)-h_n(1))
    $$
    In this case, $\omega_0=\rho$
    \begin{eqnarray*}
         A&=&\frac{M}{2n}\sum_{j=1}^{n}\frac{1}{|w_0-\omega_j|^3}+\frac{m_0}{2|w_0|^3}\\
         &=&\frac{M}{2n}\sum_{j=1}^{n}\frac{1}{((w_0|^2+|\omega_j|^2-2 Re\thinspace\omega_0 \thinspace\bar{\omega_j})^{3/2}}+\frac{m_0}{2|w_0|^3}\\
         &=&\frac{Ms^3}{2n}\sum_{j=1}^{n}\frac{1}{(1+s^2-2s\cos{\frac{2j\pi}{n}})^{3/2}}+\frac{m_0s^3}{2}\\
         &=&\frac{Ms^3}{2n}\sum_{j=1}^{n}\frac{1}{(1+s^2-2s\cos{\frac{2j\pi}{n}})^{3/2}}+\frac{Ms^3(s^3h_n(s)-h_n(1))}{2(1-s^3)}\\
         &=&\frac{Ms^3}{2n(1-s^3)}(ns^3h_n(s)-nh_n(1)+\sum_{j=1}^{n}\frac{1-s^3}{(1+s^2-2s\cos{\frac{2j\pi}{n}})^{3/2}})\\
         &=&\frac{Ms^3}{2n(1-s^3)}\Bigl(\sum_{j=1}^{n}\frac{1-s^4\cos{2j\pi\over n}}{(1+s^2-2s\cos{\frac{2j\pi}{n}})^{3/2}}-nh_n(1)\Bigr)
    \end{eqnarray*}
    and
 \begin{eqnarray*}
   B&=&\frac{3M}{2n}\sum_{j=1}^{n}\frac{(w_0-\omega_j)^2}{|w_0-\omega_j|^5}+\frac{3m_0}{2|w_0|^3}\\
   &=&\frac{3M}{2n}\sum_{j=1}^{n}\frac{\rho^2+e^{-\frac{4j\pi}{n}i}-2\rho e^{-\frac{2j\pi}{n}i}}{(\rho^2+1-2\rho\cos{\frac{2j\pi}{n}})^{5/2}}+\frac{3m_0}{2|w_0|^3}\\
   &=&\frac{3Ms^3}{2n}\sum_{j=1}^{n}\frac{1+s^2(\cos{\frac{4j\pi}{n}}-i\sin{\frac{4j\pi}{n}})-2s(\cos{\frac{2j\pi}{n}}-i\sin{\frac{2j\pi}{n}})}
   {(1+s^2-2s\cos{\frac{2j\pi}{n}})^{5/2}}+\frac{3m_0s^3}{2}\\
   &=&\frac{3Ms^3}{2n}
   \Bigl(
   \sum_{j=1}^{n}
   \frac{1+s^2\cos{\frac{4j\pi}{n}}-2s\cos{\frac{2j\pi}{n}}}
   {(1+s^2-2s\cos{\frac{2j\pi}{n}})^{5/2}}
   +i
   \sum_{j=1}^{n}
   \frac{2s\sin{\frac{2j\pi}{n}}-s^2\sin{\frac{4j\pi}{n}}}
   {(1+s^2-2s\cos{\frac{2j\pi}{n}})^{5/2}}
   +\frac{3m_0s^3}{2}\\
   &=&\frac{3Ms^3}{2n}
   \sum_{j=1}^{n}
   \frac{1+s^2\cos{\frac{4j\pi}{n}}-2s\cos{\frac{2j\pi}{n}}}
   {(1+s^2-2s\cos{\frac{2j\pi}{n}})^{5/2}}
   +\frac{3m_0s^3}{2},
    \end{eqnarray*}
     here
   $$
    \sum_{j=1}^{n}
   \frac{2s\sin{\frac{2j\pi}{n}}-s^2\sin{\frac{4j\pi}{n}}}
   {(1+s^2-2s\cos{\frac{2j\pi}{n}})^{5/2}}=0
   $$
   which can be demonstrated by the qualities of trigonometric function.
   Furthermore, if $\frac{m_0}{M}$ is sufficiently large, that is, $s$ is sufficiently near to 1, then 
   $|B|=B$.
   \\
   From above, we obtain
\begin{eqnarray*}
    \lim_{s\to 1}\frac{2A}{\omega^2}&=&\lim_{s\to 1}\frac{\frac{1}{n}\sum_{j=1}^{n}\frac{1-s^4\cos{\frac{2j\pi}{n}}}{(1+s^2-2s\cos{\frac{2j\pi}{n}})^{3/2}}-h_n(1)}
    {h_n(s)-h_n(1)}\\
    &=&\lim_{s\to 1}
    \frac{
    \frac{1}{n}\sum_{j=1}^{n}\frac{1-s^4\cos{\frac{2j\pi}{n}}}{(1+s^2-2s\cos{\frac{2j\pi}{n}})^{3/2}}-h_n(1)
    }
    {
    \frac{1}{n}\sum_{j=1}^{n}\frac{1-s\cos{\frac{2j\pi}{n}}}{(1+s^2-2s\cos{\frac{2j\pi}{n}})^{3/2}}
    -h_n(1)
    }\\
   &=&1+\lim_{s\to 1}\frac{
    \frac{1}{n}\sum_{j=1}^{n}\frac{(s-s^4)\cos{\frac{2j\pi}{n}}}{(1+s^2-2s\cos{\frac{2j\pi}{n}})^{3/2}}
   }
   {
    \frac{1}{n}\sum_{j=1}^{n}\frac{1-s\cos{\frac{2j\pi}{n}}}{(1+s^2-2s\cos{\frac{2j\pi}{n}})^{3/2}}
    -h_n(1)
   }\\
   &=&1+\lim_{s\to 1}\frac{
    \frac{1}{n}\sum_{j=1}^{n-1}\frac{(s-s^4)\cos{\frac{2j\pi}{n}}}{(1+s^2-2s\cos{\frac{2j\pi}{n}})^{3/2}}
   +\frac{1}{n}\frac{s(1-s^3)}{(1-s)^3}
   }
   {
    \frac{1}{n}\sum_{j=1}^{n-1}\frac{1-s\cos{\frac{2j\pi}{n}}}{(1+s^2-2s\cos{\frac{2j\pi}{n}})^{3/2}}
    +\frac{1}{n}\frac{1-s}{(1-s)^3}
    -h_n(1)
   }\\
   &=&4,
\end{eqnarray*}
that is,
$$
\lim_{s\to1}\frac{A}{\omega^2}=2
$$
In the meanwhile, using the similar methods, we can derive the following formula:
   \begin{eqnarray*}
       \lim_{s\to 1 }\frac{B}{\omega^2}&=& \lim_{s\to 1}
       \frac{
      \frac{3Ms^3}{2n}\sum_{j=1}^{n}
      \frac{
      1+s^2\cos{\frac{4j\pi}{n}}-2s\cos{\frac{2j\pi}{n}}
      }      {  (1+s^2-2s\cos{\frac{2j\pi}{n})} ^{5/2 } }
      +\frac{3Ms^3}{2(1-s^3)}(s^3h_n(s)-h_n(1))
      }
       {
       \frac{Ms^3}{1-s^3}(h_n(s)-h_n(1))
       }\\
       &=&\lim_{s\to 1}\frac{
       \frac{3(1-s^3)}{2n}\sum_{j=1}^{n}
      \frac{
      1+s^2\cos{\frac{4j\pi}{n}}-2s\cos{\frac{2j\pi}{n}}
      }      {  (1+s^2-2s\cos{\frac{2j\pi}{n})} ^{5/2 } }+\frac{3}{2}(s^3h_n(s)-h_n(1))
       }
       {h_n(s)-h_n(1)}\\
       &=&6.
   \end{eqnarray*}
\end{proof}

Consider the condition $\theta=\frac{\pi}{n}$ and let $s=\frac{1}{\rho}$. Through the  equilibrium equation 
    $$
    \psi_2(\rho)=m_0(1-\rho^3)+M\left(h_n\left(\frac{1}{\rho},\frac{\pi}{n}\right)-h_n(1)\rho^3\right)=0,
    $$
defined on p.309 in \cite{BaE}, we can deduce 
\begin{equation}\lb{S3}
m_0=M\frac{s^3h_n(s,\frac{\pi}{n})-h_n(1)}{1-s^3}.
\end{equation}

For the central configuration when the massless body lies at
 $S_3$, we have
\begin{proposition}\lb{A.3}
    If $\rho=S_3$,
    we have $l_3(\rho,\frac{\pi}{n})>0$ and $2A-\omega^2>0$.
    Moreover, if $\frac{m_0}{M}\to+\infty$, we have
    \begin{equation}
        l_3\to0^+,\quad \frac{A}{\omega^2}\to\left(\frac{1}{2}\right)^+.
    \end{equation}
    
    Hence, if $\frac{m_0}{M}\to+\infty$,
    we have
    \begin{equation}
        \lambda_3+\lambda_4\to3,\quad
        \lambda_4\to0^+.
    \end{equation}
\end{proposition}


\begin{proof}\label{pi/n}
From the expression of $m_0$, we obtain
$$\frac{m_0}{M}({1-s^3})=s^3{h_n(s,\frac{\pi}{n})-h_n(1)}.$$
Based on the equation above and the normalization $M+m_0=1$, we have 
$$
   \lim_{s\to 1^{-}}\frac{1-s^3}{M}=h_n(1,\frac{\pi}{n})-h_n(1).
$$
Combining the expressions of $\omega^2$, $A$ and $l_3(s,\frac{\pi}{n})$ on p.312-p.313 in \cite{BaE}, 
\begin{eqnarray*}
   \omega^2&=&\frac{Ms^3}{1-s^3}(h_n(s,\frac{\pi}{n})-h_n(1)),\\
   A &=&\frac{Ms^3}{2(1-s^3)}\left(s^3h_n\left(s,\frac{\pi}{n}\right)-h_n(1)+\frac{1-s^3}{n}\sum_{j=1}^{n}\frac{1}{(1+s^2-2s\cos({\frac{2j\pi}{n}}+\frac{\pi}{n})^{3/2}}\right),\\
   l_3(s,\frac{\pi}{n})&=&A+\omega^2-|B|=\frac{Ms^3}{2n}\sum_{j=1}^{n}\frac{1}{(1+s^2-2s\cos({\frac{2j\pi}{n}}+\frac{\pi}{n})^{3/2}}+Ms^3h_n(s,\frac{\pi}{n})\\
   & &-\frac{3Ms^3}{2n}\sum_{j=1}^{n}\frac{1+s^2\cos\left({\frac{4j\pi}{n}}+\frac{2\pi}{n}\right)-2s\cos\left({\frac{2j\pi}{n}+\frac{\pi}{n}}\right)}{(1+s^2-2s\cos{(\frac{2j\pi}{n}}+\frac{\pi}{n})^{5/2}},
\end{eqnarray*}
we can deduce $\omega^2\to 1$, $A\to \frac{1}{2}$ and $l_3(s,\frac{\pi}{n})\to 0$ as $s\to 1^{-}$.

Furthermore, 
\begin{eqnarray*}
\lim_{s\to1^{-}}\frac{l_3}{M}&=&\frac{1}{2n}\sum_{j=1}^{n}\frac{1}{(2-2\cos(\frac{2j\pi}{n}+\frac{\pi}{n}))^{\frac{3}{2}}}
+\frac{1}{n}\sum_{j=1}^{n}\frac{1-\cos(\frac{2j\pi}{n}+\frac{\pi}{n})}{(2-2\cos(\frac{2j\pi}{n}+\frac{\pi}{n}))^{\frac{3}{2}}}\\
&\;&
-\frac{3}{2n}\sum_{j=1}^{n}
\frac{1+\cos(\frac{4j\pi}{n}+\frac{2\pi}{n})-2\cos(\frac{2j\pi}{n}+\frac{\pi}{n})}
{(2-2\cos(\frac{2j\pi}{n}+\frac{\pi}{n}))^{\frac{5}{2}}}\\
&=&\newcommand{\uj}{\frac{j\pi}{n}+\frac{\pi}{2n}}
\frac{1}{2n}\sum_{j=1}^{n}\frac{1}{2^3\sin^3(\uj)}
+\frac{1}{n}\sum_{j=1}^{n}\frac{1-\cos(2(\uj))}{2^3\sin^3(\uj)}\\
&\;&\newcommand{\uj}{\frac{j\pi}{n}+\frac{\pi}{2n}}
-\frac{3}{2n}\sum_{j=1}^{n}\frac{1+\cos(4(\uj))-2\cos(2({\uj}))}{2^5\sin^5(\uj)}
\\
&=&\newcommand{\uj}{\frac{j\pi}{n}+\frac{\pi}{2n}}
\frac{1}{n}\sum_{j=1}^{n}\frac{-4\sin^4(\uj)+8\sin^2(\uj)}{2^5\sin^5(\uj)}\\
&\;&\newcommand{\uj}{\frac{j\pi}{n}+\frac{\pi}{2n}}
-\frac{3}{2n}\sum_{j=1}^{n}\frac{1+\cos4(\uj)-2\cos2(\uj)}{2^5\sin^5(\uj)}
\\
&=&\newcommand{\uj}{\frac{j\pi}{n}+\frac{\pi}{2n}}\frac{1}{n}\sum_{j=1}^{n}\frac{-4\sin^4(\uj)+8\sin^2(\uj)}{2^5\sin^5(\uj)}\\
&>&0,
\end{eqnarray*}
which implies that $\frac{l_3}{M}$ is bigger than 0 when s sufficiently near to 1. By a direct calculation, we have
$$2A-\omega^2=\frac{Ms^3}{n}\sum_{j=1}^n\frac{s \thinspace\cos(\frac{2j\pi}{n}+\frac{\pi}{n})}{(1+x^2-2x\thinspace cos(\frac{2j\pi}{n}+\frac{\pi}{n}))^{3/2}}.$$
The sum of above equation can be written as a integration of a product of two positive continuous functions(see the proof of the  Lemma 3 of \cite{BaE2})as follows:

$$\sum_{j=1}^n\frac{s \thinspace\cos(\frac{2j\pi}{n}+\frac{\pi}{n})}{(1+x^2-2x\thinspace cos(\frac{2j\pi}{n}+\frac{\pi}{n}))^{3/2}}
=\frac{1}{\pi}\int_0^1\frac{\tau^{\frac{1}{2}}(1-s^2\tau^2)}{[(1-1+\tau)(1-s^2\tau)]^{\frac{1}{2}}}$$
$$
\cdot \Biggl\{\frac{s\tau cos(\frac{2j\pi}{n}+\frac{\pi}{n})}{(1+s^2\tau^2-2s\tau cos(\frac{2j\pi}{n}+\frac{\pi}{n}))^2}\Biggr\}_n
 d\tau.$$
In  short, $ 2A-\omega^2>0.$
\end{proof}

\begin{remark}
    In the case $\theta=\frac{\pi}{n}$, there are errors in the original computations of $\omega^2$ on p.~312 and $l_3(s,\frac{\pi}{n})$ on p.~313 in \cite{BaE} by D.~Bang and B.~Elmabsout. They should be corrected as shown in the proof of \ref{pi/n}.

\end{remark}

\subsection{$\omega$-Maslov-type indices and $\omega$-Morse indices}\label{subsec:5.2}

For reader's conveniences, here we give a brief introduction on the $\om$-Maslov type
index theory, and we refer to Y. Long's book \cite{Lon4} for the history and more details
of this index theory.

As usual, the symplectic group $\Sp(2n)$ is defined by
$$ \Sp(2n) = \{M\in {\rm GL}(2n,\R)\,|\,M^TJM=J\}, $$
whose topology is induced from that of $\R^{4n^2}$. For $\tau>0$
we are interested in paths in $\Sp(2n)$:
$$ \P_{\tau}(2n) = \{\ga\in C([0,\tau],\Sp(2n))\,|\,\ga(0)=I_{2n}\}, $$
which is equipped with the topology induced from that of $\Sp(2n)$.
For any $\om\in\U$ and $M\in\Sp(2n)$, the following real function was
introduced in \cite{Lon2}:
$$ D_{\om}(M) = (-1)^{n-1}\ol{\om}^n\det(M-\om I_{2n}). $$
Thus for any $\om\in\U$ the following codimension $1$ hypersurface
in $\Sp(2n)$ is defined (\cite{Lon2}):
$$ \Sp(2n)_{\om}^0 = \{M\in\Sp(2n)\,|\, D_{\om}(M)=0\}.  $$
For any $M\in \Sp(2n)_{\om}^0$, we define a co-orientation of
$\Sp(2n)_{\om}^0$ at $M$ by the positive direction
$\frac{d}{dt}Me^{t J}|_{t=0}$ of the path $Me^{t J}$ with $0\le t\le
\varepsilon$ and $\varepsilon$ being a small enough positive number. Let
\bea
\Sp(2n)_{\om}^{\ast} &=& \Sp(2n)\bs \Sp(2n)_{\om}^0,   \nn\\
\P_{\tau,\om}^{\ast}(2n) &=&
      \{\ga\in\P_{\tau}(2n)\,|\,\ga(\tau)\in\Sp(2n)_{\om}^{\ast}\}, \nn\\
\P_{\tau,\om}^0(2n) &=& \P_{\tau}(2n)\bs \P_{\tau,\om}^{\ast}(2n). \nn\eea
For any two continuous paths $\xi$ and $\eta:[0,\tau]\to\Sp(2n)$ with
$\xi(\tau)=\eta(0)$, we define their concatenation by
$$ \eta\ast\xi(t) = \left\{\begin{matrix}
            \xi(2t), & \quad {\rm if}\;0\le t\le \tau/2, \cr
            \eta(2t-\tau), & \quad {\rm if}\; \tau/2\le t\le \tau. \cr\end{matrix}\right. $$

As in \cite{Lon4}, for $\lm\in\R\bs\{0\}$, $a\in\R$, $\th\in (0,\pi)\cup (\pi,2\pi)$,
$b=\left(\begin{matrix}b_1 & b_2\cr
                 b_3 & b_4\cr\end{matrix}\right)$ with $b_i\in\R$ for $i=1, \ldots, 4$, and $c_j\in\R$
for $j=1, 2$, we denote respectively some normal forms by
\bea
&& D(\lm)=\left(\begin{matrix}\lm & 0\cr
                         0  & \lm^{-1}\cr\end{matrix}\right), \qquad
   R(\th)=\left(\begin{matrix}\cos\th & -\sin\th\cr
                        \sin\th  & \cos\th\cr\end{matrix}\right),  \nn\\
&& N_1(\lm, a)=\left(\begin{matrix}\lm & a\cr
                             0   & \lm\cr\end{matrix}\right), \qquad
   N_2(e^{\sqrt{-1}\th},b) = \left(\begin{matrix}R(\th) & b\cr
                                           0      & R(\th)\cr\end{matrix}\right),  \nn\\
&& M_2(\lm,c)=\left(\begin{matrix}\lm &   1 &       c_1 &         0 \cr
                              0 & \lm &       c_2 & (-\lm)c_2 \cr
                              0 &   0 &  \lm^{-1} &         0 \cr
                              0 &   0 & -\lm^{-2} &  \lm^{-1} \cr\end{matrix}\right). \nn\eea
Here $N_2(e^{\sqrt{-1}\th},b)$ is {\bf trivial} if $(b_2-b_3)\sin\th>0$, or {\bf non-trivial}
if $(b_2-b_3)\sin\th<0$, in the sense of Definition 1.8.11 on p.41 of \cite{Lon4}. Note that
by Theorem 1.5.1 on pp.24-25 and (1.4.7)-(1.4.8) on p.18 of \cite{Lon4}, when $\lm=-1$ there hold
\bea
c_2 \not= 0 &{\rm if\;and\;only\;if}\;& \dim\ker(M_2(-1,c)+I)=1, \nn\\
c_2 = 0 &{\rm if\;and\;only\;if}\;& \dim\ker(M_2(-1,c)+I)=2. \nn\eea

\begin{definition}\lb{D2.1} (\cite{Lon2}, \cite{Lon4})
For any $\om\in\U$ and $M\in \Sp(2n)$, the $\om$-nullity $\nu_{\om}(M)$ is defined by
\be \nu_{\om}(M)=\dim_{\C}\ker_{\C}(M - \om I_{2n}).  \lb{A2.2}\ee

For every $M\in \Sp(2n)$ and $\om\in\U$, as in Definition 1.8.5 on p.38 of \cite{Lon4}, we define the
{\bf $\om$-homotopy set} $\Om_{\om}(M)$ of $M$ in $\Sp(2n)$ by
$$  \Om_{\om}(M)=\{N\in\Sp(2n)\,|\, \nu_{\om}(N)=\nu_{\om}(M)\},  $$
and the {\bf homotopy set} $\Om(M)$ of $M$ in $\Sp(2n)$ by
\bea  \Om(M)=\{N\in\Sp(2n)\,&|&\,\sg(N)\cap\U=\sg(M)\cap\U,\,{\it and}\; \nn\\
         &&\qquad \nu_{\lm}(N)=\nu_{\lm}(M)\qquad\forall\,\lm\in\sg(M)\cap\U\}.  \nn\eea
We denote by $\Om^0(M)$ (or $\Om^0_{\om}(M)$) the path connected component of $\Om(M)$ ($\Om_{\om}(M)$)
which contains $M$, and call it the {\bf homotopy component} (or $\om$-{\bf homtopy component}) of $M$ in
$\Sp(2n)$. Following Definition 5.0.1 on p.111 of \cite{Lon4}, for $\om\in \U$ and $\ga_i\in \P_{\tau}(2n)$
with $i=0, 1$, we write $\ga_0\sim_{\om}\ga_1$ if $\ga_0$ is homotopic to $\ga_1$ via
a homotopy map $h\in C([0,1]\times [0,\tau], \Sp(2n))$ such that $h(0)=\ga_0$, $h(1)=\ga_1$, $h(s)(0)=I$,
and $h(s)(\tau)\in \Om_{\om}^0(\ga_0(\tau))$ for all $s\in [0,1]$. We write also $\ga_0\sim \ga_1$, if
$h(s)(\tau)\in \Om^0(\ga_0(\tau))$ for all $s\in [0,1]$ is further satisfied. We write $M\approx N$, if
$N\in \Om^0(M)$.
\end{definition}

Note that we have $N_1(\lm,a)\approx N_1(\lm, a/|a|)$ for $a\in\R\bs\{0\}$ by symplectic coordinate
change, because
$$ \left(\begin{matrix}1/\sqrt{|a|} & 0\cr
                           0  & \sqrt{|a|}\cr\end{matrix}\right)
   \left(\begin{matrix}\lm & a\cr
                  0  & \lm\cr\end{matrix}\right)
   \left(\begin{matrix}\sqrt{|a|} & 0\cr
                           0  & 1/\sqrt{|a|}\cr\end{matrix}\right) = \left(\begin{matrix}\lm & a/|a|\cr
                                                                         0  & \lm\cr\end{matrix}\right). $$

Following Definition 1.8.9 on p.41 of \cite{Lon4}, we call the above matrices $D(\lm)$, $R(\th)$, $N_1(\lm,a)$
and $N_2(\om,b)$ {\bf basic normal forms} of symplectic matrices. As proved in \cite{Lon2} and \cite{Lon3} (cf.
Theorem 1.9.3 on p.46 of \cite{Lon4} and Corollary 2.3.8 on p.58), every $M\in\Sp(2n)$ has its basic normal
form decomposition in $\Om^0(M)$ as a $\dm$-sum of these basic normal forms. Here the $\dm$-sum is introduced
in the above Section 1. Indeed, we have

\begin{lemma} (\cite{Lon4}, p.58) For any $M\in\Sp(2n)$, there exist a path $f:[0,1]\to\Omega^0(M)$ and
an integer $k\ge 0$ such that
$$ f(0)=M \qquad {\it and}\qquad f(1)=M_1(\om_1)\dm\ldots\dm M_k(\om_k)\dm M_0, $$
where if $\sg(M)\cap\U\not=\emptyset$, then $k>0$ and each $M_i(\om_i)$ is a basic normal form of the
eigenvalue $\om_i\in\U$ for $1\le i\le k$;
if $\sg(M)\cap (\C\bs\U)\not=\emptyset$, then $M_0$
does appear in the above decomposition of $f(1)$ and satisfies $M_0=D(2)^{\dm p}$ or
$D(-2)\dm D(2)^{\dm(q-1)}$ for some integers $p\ge 0$ and $q\ge 1$ depending on whether the total
multiplicity of real eigenvalues of $M$ which are strictly less than $-1$ is even or odd respectively.
\end{lemma}

This lemma is very important when we derive the basic normal form decomposition of $\ga_{\bb,e}(2\pi)$ to
compute the $\om$-index $i_{\om}(\ga_{\bb,e})$ of the path $\ga_{\bb,e}$ in this paper.

We define a special continuous symplectic path $\xi_n$ in $\Sp(2n)$ by
\be \xi_n(t) = \left(\begin{matrix}2-\frac{t}{\tau} & 0 \cr
                             0 &  (2-\frac{t}{\tau})^{-1}\end{matrix}
               \right)^{\dm n} \qquad {\rm for}\;0\le t\le \tau.  \lb{A2.3}\ee

\begin{definition} (\cite{Lon4})\lb{D2.2} {For any $\tau>0$ and $\ga\in \P_{\tau}(2n)$,
the $\om$-nullity $\nu_{\om}(\ga)$ of $\ga$ is defined by
\be \nu_{\om}(\ga)= \nu_{\om}(\ga(\tau)).  \lb{A2.4}\ee

If $\ga\in\P_{\tau,\om}^{\ast}(2n)$, the $\om$-index $i_{\om}(\ga)$ of $\ga$ is defined by
\be i_{\om}(\ga) = [\Sp(2n)_{\om}^0: \ga\ast\xi_n],  \lb{A2.5}\ee
where the right hand side of (\ref{A2.5}) is the usual homotopy intersection number, and
the orientation of $\ga\ast\xi_n$ is its positive time direction under homotopy with
fixed end points.

If $\ga\in\P_{\tau,\om}^0(2n)$, we let $\mathcal{F}(\ga)$ be the set of all open
neighborhoods of $\ga$ in $\P_{\tau}(2n)$, and the $\om$-index $i_{\om}(\ga)$ of $\ga$ is
defined by
\be i_{\om}(\ga) = \sup_{U\in\mathcal{F}(\ga)}\inf\{i_{\om}(\beta)\,|\,
                       \beta\in U\cap\P_{\tau,\om}^{\ast}(2n)\}.      \lb{A2.6}\ee
Then
$$ (i_{\om}(\ga), \nu_{\om}(\ga)) \in \Z\times \{0,1,\ldots,2n\}, $$
is called the index function of $\ga$ at $\om$. }
\end{definition}

The $\omega$-index function has the following properties:

\begin{lemma} (\cite{Lon4}, pp.147-148)
For $\omega\in\U$ and $\tau>0$, the $\omega$-index part of the index function defined on paths
in $\P_\tau(2n)$ is uniquely determined by the following five axioms:

$1^\circ$ ({\bf Homotopy invariant}) For $\gamma_0$ and $\gamma_1\in\P_\tau(2n)$,
if $\gamma_0\sim_\omega\gamma_1$ on $[0,\tau]$, then
$$ i_\omega(\gamma_0)=i_\omega(\gamma_1). $$

$2^\circ$ ({\bf Symplectic additivity}) For any $\gamma_j\in\P_\tau(2n_j)$ with $j=0,1$,
$$ i_\omega(\gamma_0\diamond\gamma_1) = i_\omega(\gamma_0)+i_\omega(\gamma_1). $$

$3^\circ$ ({\bf Clockwise continuity}) For any $\ga\in\P_\tau(2)$ satisfying $\ga(\tau)=N_1(\om,b)$
with $b=\pm 1$ or $0$ when $\omega=\pm 1$, or $\ga(\tau)=R(\vf)$ with $\om=e^{\sqrt{-1}\vf}\in \U\bs\R$,
there exists a $\th_0>0$ such that
$$ i_{\om}([\ga(\tau)\phi_{\tau,-\th}]*\ga) = i_{\om}(\ga),\quad\forall 0<\theta\le\th_0, $$
where $\phi_{\tau,\th}$ is defined by
$$ \phi_{\tau,\th}=R\left(\frac{t\th}{\tau}\right),\quad\forall t\in [0,\tau],\;\th\in\R. $$

$4^\circ$ ({\bf Counterclockwise jumping}) For any $\ga\in\P_\tau(2)$ satisfying $\ga(\tau)=N_1(\om,b)$
with $b=\pm 1$ or $0$ when $\om=\pm 1$, or $\ga(\tau)=R(\vf)$ with $\om=e^{\sqrt{-1}\vf}\in \U\bs\R$,
there exists a $\th_0>0$ such that
$$ i_{\om}([\ga(\tau)\phi_{\tau,\th}]*\ga) = i_{\om}(\ga)+1,\quad\forall 0<\th\le\th_0. $$

$5^\circ$ ({\bf Normality}) For $\hat{\aa}_0(t)=D\left(1+\frac{t}{\tau}\right)$ with $0\le t\le \tau$,
$$ i_{\om}(\hat{\aa}_0)=0. $$
\end{lemma}

For paths in $\Sp(2)$, we have

\begin{lemma} (\cite{Lon4}, pp 179-183)
Let $\ga\in\mathcal{P}_\tau(2)$. Then one and only one of the following cases must happen.

$1^\circ$ If $\sg(\ga(\tau))=\{1,1\}$, then $\ga(\tau)\approx\left(\begin{matrix}1 & a\cr 0 & 1\end{matrix}\right)$
and we have
$$i_1(\ga)\in\left\{
\begin{array}{l}
    2\Z+1,{\rm\;if\;} a\ge0,\\
    2\Z,\quad\;\;{\rm\;if\;} a<0.
\end{array}\right.
$$

$2^\circ$ If $\sg(\ga(\tau))=\{-1,-1\}$, we have
$$ i_1(\ga)\in2\Z+1. $$

$3^\circ$ If $\sg(\ga(\tau))\cap\U=\emptyset$, we have
$$ i_1(\ga)\in 2\Z+\aa(\ga(\tau)), $$
where $\aa(\ga(\tau))$ is the hyperbolic index given by Definition 1.8.1 of \cite{Lon4},
i.e., $\aa(\ga(\tau))=0$ if $\ga(\tau)\approx D(\lm)$, $\aa(\ga(\tau))=1$ if $\ga(\tau)\approx D(-\lm)$
for $\lm\in(0,1)\cup (1,+\infty)$.

$4^\circ$ If $\sg(\ga(\tau))\in\U\bs\{1,-1\}$, we have
$$ i_1(\ga)\in 2\Z+1. $$
\end{lemma}

\begin{definition} (\cite{Lon2}, \cite{Lon4})\lb{D2.3}
For any $M\in\Sp(2n)$ and $\om\in\U$, choosing $\tau>0$ and $\ga\in\P_\tau(2n)$ with $\ga(\tau)=M$,
we define
\be
S_M^{\pm}(\om)=\lim_{\epsilon\rightarrow0^+}\;i_{\exp(\pm\epsilon\sqrt{-1})\om}(\ga)-i_\om(\ga).
\ee
They are called the splitting numbers of $M$ at $\om$.
\end{definition}
The splitting numbers $S_M^{\pm}(\om)$ measures the jumps between $i_\om(\ga)$
and $i_\lambda(\ga)$ with $\lambda\in\U$ near $\om$ from two sides of $\om$ in $\U$.
For any $\om_0=e^{\sqrt{-1}\th_0}\in\U$ with $0\le\th_0<2\pi$,
we denote by $\om_j$ with $1\le j\le p_0$ the eigenvalues of $M$ on $\U$ which are
distributed counterclockwise from $1$ to $\om_0$ and located strictly between $1$ and $\om_0$.
Then we have
\be i_{\om_0}(\ga)=i_1(\ga)+S_M^+(1)+\sum_{j=1}^{p_0}(-S_M^-(\om_j)+S_M^+(\om_j))-S_M^-(\om_0).  \lb{5.A1}\ee

\begin{lemma} (\cite{Lon4}, p.198)
The integer valued splitting number pair $(S_M^+(\om),S_M^-(\om))$ defined for all
$(\om,M)\in\U\times\cup_{n\ge1}\Sp(2n)$ are uniquely determined by the following axioms:

$1^{\circ}$ (Homotopy invariant) $S_M^{\pm}(\om)=S_N^{\pm}(\om)$ for all $N\in\Omega^0(M)$.

$2^{\circ}$ (Symplectic additivity) $S_{M_1\diamond M_2}^{\pm}(\om)=S_{M_1}^{\pm}(\om)+S_{M_2}^{\pm}(\om)$
for all $M_i\in\Sp(2n_i)$ with $i=1$ and $2$.

$3^{\circ}$ (Vanishing) $S_M^{\pm}(\om)=0$ if $\om\not\in\sigma(M)$.

$4^{\circ}$ (Normality) $(S_M^+(\om),S_M^-(\om))$ coincides with the ultimate type of
$\om$ for $M$ when $M$ is any basic normal form.
\end{lemma}

Moreover, by Lemma 9.1.6 on p.192 of \cite{Lon4} for $\om\in\C$ and $M\in\Sp(2n)$, we have
\be  S_M^+(\om)=S_M^-(\overline\om).  \ee

Note that here the ultimate type of $\om\in\U$ for a symplectic matrix $M$ mentioned in the above lemma
is given in Definition 1.8.12 on pp.41-42 of \cite{Lon4} algebraically with its more properties studied
there.

For the reader's conveniences, following List 9.1.12 on pp.198-199 of \cite{Lon4}, we list below
the splitting numbers (i.e., the ultimate types) for all basic normal forms.

\begin{lemma}\label{Lm:splitting.numbers}
We have:

$\langle$1$\rangle$ $(S_M^+(1),S_M^-(1))=(1,1)$ for $M=N_1(1,b)$ with $b=1$ or $0$.

$\langle$2$\rangle$ $(S_M^+(1),S_M^-(1))=(0,0)$ for $M=N_1(1,-1)$.

$\langle$3$\rangle$ $(S_M^+(-1),S_M^-(-1))=(1,1)$ for $M=N_1(-1,b)$ with $b=-1$ or $0$.

$\langle$4$\rangle$ $(S_M^+(-1),S_M^-(-1))=(0,0)$ for $M=N_1(-1,1)$.

$\langle$5$\rangle$ $(S_M^+(e^{\sqrt{-1}\th}),S_M^-(e^{\sqrt{-1}\th}))=(0,1)$
for $M=R(\theta)$ with $\th\in(0,\pi)\cup(\pi,2\pi)$.

$\langle$6$\rangle$ $(S_M^+(\om),S_M^-(\om))=(1,1)$ for $M=N_2(\om,b)$
being non-trivial (cf. Definition 1.8.11 on p.41 of \cite{Lon4}) with
$\om=e^{\sqrt{-1}\th}\in\U\backslash\R$.

$\langle$7$\rangle$ $(S_M^+(\om),S_M^-(\om))=(0,0)$ for $M=N_2(\om,b)$
being trivial (cf. Definition 1.8.11 on p.41 of \cite{Lon4}) with
$\om=e^{\sqrt{-1}\th}\in\U\backslash\R$.

$\langle$8$\rangle$ $(S_M^+(\om),S_M^-(\om))=(0,0)$ for $\om\in\U$ and $M\in\Sp(2n)$
satisfying $\sigma(M)\cap\U=\emptyset$.
\end{lemma}

\medskip

We refer to \cite{Lon4} for more details on this index theory of symplectic matrix paths
and periodic solutions of Hamiltonian system.

\medskip

For $T>0$, suppose $x$ is a critical point of the functional
$$ F(x)=\int_0^TL(t,x,\dot{x})dt,  \qquad \forall\,\, x\in W^{1,2}(\R/T\Z,\R^n), $$
where $L\in C^2((\R/T\Z)\times \R^{2n},\R)$ and satisfies the
Legendrian convexity condition $L_{p,p}(t,x,p)>0$. It is well known
that $x$ satisfies the corresponding Euler-Lagrangian
equation:
\bea
&& \frac{d}{dt}L_p(t,x,\dot{x})-L_x(t,x,\dot{x})=0,    \label{A2.7}\\
&& x(0)=x(T),  \qquad \dot{x}(0)=\dot{x}(T).    \label{A2.8}\eea

For such an extremal loop, define
\bea
P(t) &=& L_{p,p}(t,x(t),\dot{x}(t)),  \nn\\
Q(t) &=& L_{x,p}(t,x(t),\dot{x}(t)),  \nn\\
R(t) &=& L_{x,x}(t,x(t),\dot{x}(t)).  \nn\eea
Note that
\be F\,''(x)=-\frac{d}{dt}(P\frac{d}{dt}+Q)+Q^T\frac{d}{dt}+R. \lb{A2.9}\ee

For $\omega\in\U$, set
\be  D(\omega,T)=\{y\in W^{1,2}([0,T],\C^n)\,|\, y(T)=\omega y(0) \}.   \lb{A2.10}\ee
We define the $\omega$-Morse index $\phi_\omega(x)$ of $x$ to be the dimension of the
largest negative definite subspace of
$$ \langle F\,''(x)y_1,y_2 \rangle, \qquad \forall\;y_1,y_2\in D(\omega,T), $$
where $\langle\cdot,\cdot\rangle$ is the inner product in $L^2$. For $\omega\in\U$, we
also set
\be  \ol{D}(\omega,T)= \{y\in W^{2,2}([0,T],\C^n)\,|\, y(T)=\omega y(0), \dot{y}(T)=\om\dot{y}(0) \}.
                     \lb{A2.11}\ee
Then $F''(x)$ is a self-adjoint operator on $L^2([0,T],\R^n)$ with domain $\ol{D}(\omega,T)$.
We also define the $\om$-nullity $\nu_{\om}(x)$ of $x$ by
$$ \nu_{\om}(x)=\dim\ker(F''(x)).  $$
Note that we only use $n=2$ in (\ref{A2.11}) from Subsection 2.2 to Section 4.

In general, for a self-adjoint linear operator $A$ on the Hilbert space $\mathscr{H}$, we set
$\nu(A)=\dim\ker(A)$ and denote by $\phi(A)$ its Morse index which is the maximum dimension
of the negative definite subspace of the symmetric form $\langle A\cdot,\cdot\rangle$. Note
that the Morse index of $A$  is equal to the total multiplicity of the negative eigenvalues
of $A$.

On the other hand, $\td{x}(t)=(\partial L/\partial\dot{x}(t),x(t))^T$ is the solution of the
corresponding Hamiltonian system of (\ref{A2.7})-(\ref{A2.8}), and its fundamental solution
$\gamma(t)$ is given by
\bea \dot{\gamma}(t) &=& JB(t)\gamma(t),  \lb{A2.12}\\
     \gamma(0) &=& I_{2n},  \lb{A2.13}\eea
with
\be B(t)=\left(\begin{matrix}P^{-1}(t)& -P^{-1}(t)Q(t)\cr
                       -Q(t)^TP^{-1}(t)& Q(t)^TP^{-1}(t)Q(t)-R(t)\end{matrix}\right). \lb{A2.14}\ee

\begin{lemma} (\cite{Lon4}, p.172)\lb{Lm:A.12}  
For the $\omega$-Morse index $\phi_\omega(x)$ and nullity $\nu_\omega(x)$ of the solution $x=x(t)$
and the $\omega$-Maslov-type index $i_\omega(\gamma)$ and nullity $\nu_\omega(\gamma)$ of the
symplectic path $\ga$ corresponding to $\td{x}$, for any $\omega\in\U$ we have
\be \phi_\omega(x) = i_\omega(\gamma), \qquad \nu_\omega(x) = \nu_\omega(\gamma).  \lb{A2.15}\ee
\end{lemma}

For more information on these topics, we refer to \cite{Lon4}.

\end{document}